\documentclass[11pt]{amsart}
\usepackage[utf8]{inputenc}
\usepackage[T1]{fontenc}
\usepackage{amsmath,amssymb,amsfonts,amsthm}
\usepackage{mathtools}
\usepackage{mathrsfs}
\usepackage{dsfont}
\usepackage{enumitem}
\usepackage{multicol}
\usepackage{graphicx}
\usepackage{float}      
\usepackage{graphicx}
\usepackage{subcaption}
\usepackage{rotating}
\usepackage{lscape}
\usepackage[table]{xcolor}
\definecolor{lightgray}{gray}{0.95}
\usepackage{hyperref}
\usepackage{geometry}
\theoremstyle{plain}
\newtheorem{theorem}{Theorem}[section]
\newtheorem{lemma}[theorem]{Lemma}
\newtheorem{proposition}[theorem]{Proposition}
\newtheorem{coro}[theorem]{Corollary}
\newtheorem{definition}[theorem]{Definition}
\newtheorem{example}[theorem]{Example}
\newtheorem{eje}[theorem]{Example}        
\newtheorem{remark}[theorem]{Remark}
\newtheorem{obs}[theorem]{Observation}

\newcommand{\ex}{{\rm e}}

\newcommand{\dps}{\displaystyle}

\setlist{itemsep=2pt, topsep=4pt, parsep=2pt}

\title[Yet Another Characterisation of Classical Orthogonal Polynomials?]{Yet Another Characterisation of Classical Orthogonal Polynomials?}

\author{K. Castillo}
\address{CMUC, Department of Mathematics, University of Coimbra, 3000-143 Coimbra,
Portugal}
\email{kenier@mat.uc.pt}

\author{G. Gordillo-N\'u\~nez}
\address{CMUC, Department of Mathematics, University of Coimbra, 3000-143 Coimbra,
Portugal}
\email{up202310693@up.pt}

\date{\today}

\begin{document}

\begin{abstract}
The \emph{NIST Handbook of Mathematical Functions} (2010) and the \emph{NIST Digital Library of Mathematical Functions} (2025) classify classical orthogonal polynomials through Bochner’s 1929 algebraic-differential characterisation and its discretisation. Yet this classification rests on a narrow reading of Bochner’s work and on a restricted notion of orthogonality that becomes inadequate once polynomials are characterised by their algebraic properties. As a result, algebraically equivalent families are treated as distinct, parameter domains are restricted, and families already implicit in Bochner’s scheme are excluded. In the mid-1980s, Maroni challenged this view by extending the notion of classical orthogonal polynomials through duality theory on locally convex spaces, thereby reaching the algebraic limits latent in Bochner’s framework. Yet when the notion was later enlarged to include further families, Maroni’s criteria and rationale were largely set aside. To clarify this history, we revisit a less familiar line of development and use it to obtain a classification of classical orthogonal polynomials on linear lattices within Maroni’s functional-analytic setting, beyond the positive-definite case. This classification recovers all known families as special cases, preserves orthogonality and the defining algebraic properties, places supposedly new families in their proper structural context, and shows that algebraically identical polynomials are often treated as distinct. Moreover, through a limit process in the weak topology of the continuous dual, we recover families implicit in Bochner’s work and unify the continuous and discrete cases within a dual-topological framework. Thus, neither Bochner’s classical characterisation nor its discrete analogue is modified to produce ad hoc families; both are recovered at the level of their intrinsic algebraic structure.
\end{abstract}

\maketitle
\setcounter{tocdepth}{2} 
\tableofcontents
\bigskip

\section{Introduction}\label{sec1}

Anyone with a modest background in mathematical physics, and in particular in the
theory of special functions, will almost invariably associate the term
\emph{classical orthogonal polynomials} with the Jacobi and Laguerre families,
with parameters implicitly assumed to be real and greater than $-1$, as well as with
the Hermite polynomials, typically following the presentation in the seminal
monograph by Szeg\H{o}~\cite[Section~2.4]{S75}, whose first edition appeared in 1939 and
which remains, in many respects, irreplaceable. This agrees with the most recent presentation in the \emph{NIST Digital Library of Mathematical
Functions} (DLMF), \S\,18.3\footnote{Version~1.2.5, released on
December~15,~2025, available at \url{https://dlmf.nist.gov/18.3}.}, as well as with
the corresponding treatment in the \emph{NIST Handbook of Mathematical Functions},
published by Cambridge University Press in 2010 and intended to replace its
predecessor, the 1964 National Bureau of Standards reference work
\emph{Handbook of Mathematical Functions with Formulas, Graphs, and Mathematical
Tables}. In the absence of further clarification, it is historically well founded to regard these families as the classical orthogonal polynomials. However, if the rationale underlying this definition is understood as
technical rather than historical, it is important to observe that it is
grounded in positive Borel measures on $\mathbb{R}$, in what is commonly referred to as the theory
of orthogonal polynomials on the real line (OPRL). This perspective more than satisfied the needs of Szegő’s intended audience, shaped by a mathematical culture markedly different from our own, particularly given that the principal applications of the time did not call for a broader framework and that the mathematical tools then available would not, in any case, have permitted a substantial development on a fully rigorous footing.
Nevertheless, the OPRL framework leads to an unduly restrictive conceptual stance when attention is directed toward the algebraic properties of polynomial families. Such properties are not ancillary but lie at the very core of any genuine characterisation of these families. This limitation manifests itself, on the one hand, through an unwarranted restriction of the admissible parameter ranges for the Jacobi and Laguerre polynomials and, on the other, through the relegation of the Bessel polynomials to a position that effectively excludes them from any claim to classical status in this context.

Quite apart from Szegő’s treatment, it is both historically and mathematically significant to note that Bochner had already made the point in his six-and-a-half-page note from 1929~\cite{B29} that the parameters of the Jacobi and Laguerre polynomials need not be subject to very restrictive conditions in order to arise as solutions to problems of practical and theoretical relevance, and that the Bessel polynomials share a common origin with the Jacobi, Laguerre, and Hermite families, namely, they all arise as polynomial solutions \( p_n \) of exact degree \( n \), for every \( n \in \mathbb{N}^{\times}=\{1,2,\dots\} \), to the Sturm-Liouville equation\footnote{In Bochner’s original formulation of the problem, the functions $\phi$ and $\psi$ are not required to be polynomials and may be taken to be as general as reasonably possible; nevertheless, as Bochner himself shows, the problem ultimately collapses to the framework considered here.}
\begin{equation}\label{bochner} 
\phi(x)\,\frac{\mathrm{d}^2 y}{\mathrm{d}x^2}(x) + \psi(x)\,\frac{\mathrm{d}y}{\mathrm{d}x}(x) + \lambda_n\, y = 0,
\end{equation}
where $\phi$ is a nonzero polynomial of degree at most two, $\psi$ is a polynomial of degree exactly one, both independent of $n$, and $\lambda_n$ is a nonzero constant depending on $n$. The simplicity of the result led Bochner to wonder whether it might already be found somewhere. Yet, as he himself remarked---``{[...] since this classification does not appear to be found in the relevant literature, we shall allow ourselves to elaborate on it [...]}''. In fact, as early as 1927, Brenke had already touched upon this issue in a work that would only be published three years later \cite{B30}, in which he classified the only OPRL solutions of equation \eqref{bochner} and obtained what Szegő would later refer to as classical orthogonal polynomials\footnote{Szegő’s book makes no reference to Brenke’s work in any of its editions, whereas Bochner’s contribution is cited as early as the first edition.}. What distinguishes Bochner’s work, however, is its independence from orthogonality considerations, a feature that ultimately led to the emergence of a new and structurally rich family, later known as the Bessel polynomials. Table~\ref{table} below\footnote{The values of \(\lambda_n\) have been deliberately omitted, since they play no role in the way we shall understand the equation \eqref{bochner} (see \eqref{pearson} below), and can be easily computed when needed.} displays the only possible polynomial solutions to equation \eqref{bochner} up to linear transformations of the variable, a point explicitly discussed as well in \cite{B29}, while the last column displays the conditions established therein, which are subject only to the requirement that the solutions form a basis for the space \(\mathbb{C}[x]\) of all complex polynomials in a single variable \(x\).  At this point, the reader may wonder whether the families listed in Table~\ref{table}, when considered with the full freedom of parameters allowed by Bochner, are in fact orthogonal. The short answer is that each of these families is orthogonal with respect to a suitable notion of orthogonality, and that Hermite, Laguerre, Bessel, and Jacobi are orthogonal with respect to the same one.

\begin{table}[h!]
\centering
\begin{tabular}{>{\columncolor{lightgray}}l|l|l|l|l}
\rowcolor{lightgray}
Family 
& Monic 
& $\phi(x)$ 
& $\psi(x)$ 
& Conditions on $\alpha, \beta \in \mathbb{C}$ \\ \hline
\rule{0pt}{18pt}Hermite     
& $H_n(x)$
& $1$ 
& $-2x$ 
& not applicable 
\\
\rule{0pt}{20pt}Laguerre    
& $L_n^{(\alpha)}(x)$
& $x$  
& $-x+\alpha+1$ 
& $-\alpha\notin \mathbb{N}^{\times}$ 
\\
\rule{0pt}{20pt}Bessel      
& $B_n^{(\alpha)}(x)$
& $x^2$  
& $(\alpha+2)x+2$ 
& $-(\alpha+1)\notin \mathbb{N}^{\times}$  
\\
\rule{0pt}{20pt}Jacobi      
& $P_n^{(\alpha,\beta)}(x)$
& $-x^2+1$ 
& $-(\alpha+\beta+2)x-\alpha+\beta$ 
& $-\alpha, -\beta, -(\alpha+\beta+1)\notin \mathbb{N}^{\times}$  
\\
\rule{0pt}{20pt}Monomials  
& $x^n$
& $x^2$ 
& $\alpha x$ 
& none  
\\
\end{tabular}
\caption{The polynomial solutions discussed in Bochner’s 1929 paper.}
\label{table}
\end{table}
 
In the early 1940s, Geronimus~\cite{G40} and Krall~\cite{K41} showed that,
under a broader notion of orthogonality, reflecting the transition toward
a functional-analytic viewpoint, and articulated more clearly in the
work of Geronimus, albeit at an early and still somewhat immature
stage, the four principal families identified by Bochner, with the full
freedom of parameters allowed in his formulation, are precisely the only
orthogonal polynomial sequences whose derivatives are again orthogonal
polynomial sequences. Despite this, the entrenched requirement that orthogonality be tied to a measure was so dominant that the Bessel polynomials received proper attention, and ultimately the name they fully deserve, only in 1949, when Krall and Frink identified a complex non-positive measure supported on a curve in \(\mathbb{C}\) with respect to which they are orthogonal \cite{KF49}. In their work, they state that ``they [the Bessel polynomials] are orthogonal, but not in quite the same sense as the other three systems'', a remark that may strike the reader as surprising in light of what we have asserted above. This, however, as we shall see, is entirely tied to a Hilbertian compression of the problem, rather than to a broader locally convex dual framework, in which Hilbert space realizations arise as special cases.

Contrary to the trends that had begun to emerge in the preceding decade, the 1950s witnessed a significant proliferation of results largely based on a restricted reading of Bochner’s work through the lens of OPRL. This orientation, whose influence extends well into the present day, was nonetheless formally questioned in the appendix to the 1961 Russian translation of Szegő’s book~\cite{S75}, subsequently translated into English by Boas and published by the American Mathematical Society in 1977~\cite{G77}. In that text, Geronimus advocated for a broader framework, capable of capturing the many properties of orthogonal polynomials that arise from purely formal methods, and proposed a notion of orthogonality based on linear functionals. The starting point was a sequence of complex numbers \((a_n)_{n \in \mathbb{N}}\) such that the Hankel determinants \(\|a_{j+k}\|_{j,k=0}^{n}\) are nonzero for every \(n\in\mathbb{N}\), an idea that long preceded the early-1940s contributions of Geronimus and Krall previously discussed. Using these values, Geronimus uniquely defined\footnote{See also the 1938 monograph by Akhiezer and Krein on moment theory, whose English translation by Fleming and Prill appeared in 1962 under the auspices of the American Mathematical Society \cite{AK62}.} a regular or quasi-definite\footnote{When all the Hankel determinants are positive, the associated linear functional is positive definite.} linear functional \(\mathbf{a}\) on \(\mathbb{C}[x]\)  by prescribing its values on the (Hamel) basis \((x^n)_{n \in \mathbb{N}}\), namely,
\begin{align}\label{moments}
\mathbf{a}(x^n)= a_n.
\end{align}
Accordingly, a sequence of polynomials \((p_n)_{n \in \mathbb{N}}\)  is said to be orthogonal with respect to the sequence \((a_n)_{n \in \mathbb{N}}\), or, more succinctly, orthogonal with respect to a regular linear functional \(\mathbf{a}\), whenever
\begin{align}\label{orth}
\mathbf a(p_n\, p_m)=
\begin{cases}
\neq 0, & n = m, \\[6pt]
0,      & n \neq m.
\end{cases}
\end{align}
Geronimus’s comments do not involve any explicit topological considerations, but his arguments can naturally be recast within the framework of modern functional analysis.  In fact, by the time Geronimus’s appendix appeared in Russian, the foundations of locally convex spaces (LCS), along with their corresponding duality theory, were already well established; and by the time its English translation became available, these ideas had even become a standard part of the curriculum in many universities worldwide. It is this broader perspective that underwrites the so-called \emph{algebraic theory of orthogonal polynomials}\footnote{The term algebraic is employed to highlight an approach focused on the algebraic properties of orthogonal polynomials, without implying a formal algebraic context.}. Within this framework, the OPRL formulation corresponds to the special situation in which the underlying functional is positive definite and admits a representation by a measure. The foundations of this theory were laid by Maroni~\cite{M84, M88, M91a, M91b}, a former student of de Possel with a strong influence from the Bourbaki tradition of his time, and who discussed these matters personally with Dieudonné in the mid-1980s. Maroni also remarked, in a personal communication with the first author, that such ideas were, at the time, entirely natural within the academic milieu in which he was working: the Laboratoire d’Analyse Numérique at the Université Pierre et Marie Curie-Paris VI (now the Laboratoire Jacques-Louis Lions). 

By defining, among other operations, via transposition, the operations of multiplying a polynomial \(p_0\) by a polynomial \(q\) and differentiating \(q\), and denoting these by \(p_0 \mathbf{a} \) and \(\mathrm{D}\mathbf{a}\), respectively, and by applying the standard procedure of duality in the setting of LCS\footnote{It is worth emphasizing that the present notion, standard in mathematics, should not be confused with the concept of duality as it is sometimes used in the theory of orthogonal polynomials \cite[Definition 3.1]{KLS10}.}, Maroni observed that a sequence of orthogonal polynomials solves equation \eqref{bochner} if and only if it is orthogonal with respect to a regular linear functional \(\mathbf{a}\) satisfying the equation
\begin{align}\label{pearson}
\mathbf{D}(\phi \mathbf{a}) = \psi \mathbf{a},
\end{align}
where $\phi$ is a nonzero polynomial of degree at most two and $\psi$ is a polynomial of degree exactly one, both independent of $n$. Consequently, up to translations and homotheties induced via transposition from such operations, the only regular linear functionals that solve equation \eqref{pearson} are those associated with the Jacobi, Laguerre, Hermite, and Bessel polynomials listed in Table \ref{table}. This not only reformulates Bochner’s problem in a modern and functionally adaptable language, but also unifies the foundational equations of Bochner, Geronimus, and Krall, now detached from the moment-based setting, showing that Bochner’s principal families are not only orthogonal under the same generality as in his original work, but are in fact orthogonal in exactly the same sense. (It is elementary to observe that a monomial sequence fails to be orthogonal in the sense of \eqref{orth}, despite being orthogonal with respect to the Lebesgue measure on the unit circle.) In summary, for Maroni, the classical orthogonal polynomials were precisely the principal Bochner families, four, not three, as he liked to quip, recalling the well-known paradox of the Three Musketeers. Perhaps one example that captures the mindset of the period can be seen by looking at a 1985 paper by Maroni on semiclassical orthogonal polynomials, which opens with the sentence {``Les polyn\^omes classiques (Jacobi, Bessel, Laguerre, Hermite) peuvent \^etre caract\'eris\'es de plusieurs fa\c{c}ons. En particulier, ce sont les suites orthogonales dont la suite des d\'eriv\'ees est aussi orthogonale.'' and cites the work of Krall and Geronimus. If we read the zbMATH review (Zbl~0591.33009) of that paper, Askey begins with the following sentence: ``The only orthogonal polynomials whose derivatives are also orthogonal are the classical polynomials of Jacobi, Laguerre and Hermite.''}

Maroni's viewpoint was formalised explicitly in his celebrated work \emph{Une théorie algébrique des polynômes orthogonaux. Applications aux polynômes orthogonaux semiclassiques} \cite[Section~6]{M91a}. In a note published in the Comptes Rendus de l’Académie des Sciences de Paris \cite{M91b}, Maroni pointed out that, within his approach, the fundamental properties of the classical families follow trivially by duality. The note appeared in 1991 and was presented by Dieudonné; a more complete and detailed treatment was later published in \cite{M93b}, while \cite{M94} largely revisits the same material in a survey‑style presentation. Over time, several results originating in these works were later credited to other
authors in a non-negligible part of the literature, a process through which they
became crystallized in the oral tradition and ultimately entered some of the most
authoritative references in analysis. In this broader historical context,
Simon points out in his widely respected five-volume work \emph{A Comprehensive
Course in Analysis}~\cite[p.~255]{S15IV} that ``there are a number of
Bochner-Brenke-like theorems in the literature,'' referring in particular to
\cite{MBP91}, where the authors ``find all sets of polynomials obeying a
differential equation and orthogonal with respect to a set of moments, without
establishing when those moments arise from a positive measure.'' Without entering into a critical analysis of \cite{MBP91}, and while noting that its results fall within the scope of Maroni’s earlier works, we simply observe that this description fits naturally within the framework developed in \cite[Section~6]{M91a}. Such a perspective does not appear in \cite{MBP91}, where the domain of the parameters---namely, the regularity of the functional---is not analyzed and the discussion is instead restricted to a Hilbertian view of the problem (see \cite[Table~I]{MBP91}). In \cite[Section~6]{M91a} the connection with Bochner’s work is also made explicit and is described as ``une manière de caractériser les suites classiques, sans doute la plus ancienne.''

In view of the preceding discussion, it becomes difficult, for instance, to
persuade a reader free of preconceived notions that the Laguerre polynomials
appearing in Bochner’s work,
\[
L_n^{(\alpha)}(x)
= \sum_{k=0}^\infty \binom{n+\alpha}{\,n-k\,}\frac{(-x)^k}{k!},
\]
which, as algebraic objects, are well defined for $-\alpha \notin \mathbb{N}^{\times}$ and are orthogonal with respect to a linear functional for these values of $\alpha$, should be regarded as \emph{classical orthogonal polynomials} only
under the additional restriction $\alpha>-1$. The issue becomes even more puzzling for a hypothetical reader when one observes that the three properties highlighted in the \emph{NIST DLMF} as the most important characterisations of what are termed classical orthogonal polynomials---namely, Bochner’s characterisation~\cite{B29}, the results of Geronimus~\cite{G40} and Krall~\cite{K41}, and the Rodrigues formula, which Maroni showed to remain valid in his broader setting (see the remark at the end of~\cite[Section~6]{M91a})---are all algebraic properties themselves and, moreover, hold for the Hermite, Jacobi, Laguerre, and Bessel polynomials in the full generality originally considered by Bochner. This point, however, is not made explicit in the \emph{NIST DLMF}, where one finds only a brief remark, echoing the comment of Krall and Frink from 1949, intended to justify the exclusion of the Bessel polynomials:
“Bessel polynomials are often included among the classical OP’s. However, in
general they are not orthogonal with respect to a positive measure, but a finite
system has such an orthogonality.” The question, ultimately, is how far we are willing to subordinate the intrinsic nature of the polynomials in order to preserve a notion of orthogonality that has become increasingly strained, particularly when the very term ``classical'' has, in recent decades, lost nearly all historical justification, having been stretched to accommodate disparate families. Do not the works of Geronimus, Krall, and Maroni already point toward the possibility of a more principled framework, one that remains faithful to the algebraic structure of these polynomials, one that privileges expansion over exclusion, and one that traces a coherent line from Bochner’s legacy to the present day?

At this point, it may be helpful to be more explicit, since, in our view, some aspects of the recent literature may have contributed to ambiguities in the interpretation of this notion. In particular, consulting what is widely and deservedly regarded as the most authoritative reference on the subject, even though its point of view is not ours, the monograph {\em Hypergeometric Orthogonal Polynomials and Their q-Analogues}~\cite{KLS10}, one sees that Theorem~4.1 classifies six families as classical, seemingly in relation to Bochner’s framework. Three of these are, in essence, the same as those already identified by Brenke, while three additional families appear as finite sequences of polynomials and are referred to as finite Jacobi polynomials, finite Bessel polynomials, and finite pseudo-Jacobi polynomials. From a structural point of view, the theory implies that all admissible families are already encompassed by Bochner’s framework, and that enlarging this class beyond four necessarily requires imposing additional restrictions, such as those inherent to the OPRL setting. In~\cite[Theorem~4.1]{KLS10},
these families are presented via their associated weight functions, a feature
that renders immediate identification nontrivial. Nevertheless, following the approach developed in~\cite[Section~8.2]{CP25}, and
adopting the case-by-case subdivision of~\cite[Theorem~4.1]{KLS10}, we denote the
corresponding families by
\begin{align*}
&(P_n^{(I)})_{n\ge 0},\quad
(P_n^{(II)})_{n\ge 0},\quad
(P_n^{(IIIa)})_{n\ge 0},\\[7pt]
&(P_n^{(IIIb)})_{n=0}^{N},\quad
(P_n^{(IIIc)})_{n=0}^{N},\quad
(P_n^{(IIId)})_{n=0}^{N},\quad
N\in \mathbb{N}^{\times},
\end{align*}
each of which can then be explicitly identified with one of the families listed in Table~\ref{table}.  The computations involved are straightforward but lengthy and are therefore omitted.  We thus record directly the explicit expressions of the corresponding families, respecting the notation for the parameters adopted in~\cite[Theorem~4.1]{KLS10}. One has
\begin{align*}
P_n^{(I)}(x)
&= \frac{1}{(-\epsilon)^{\tfrac{n}{2}}}
  H_n\!\left(\sqrt{-\epsilon}\left(x+\tfrac{\gamma}{2\epsilon}\right)\right),\quad P_n^{(II)}(x)
= \frac{1}{(-2\epsilon)^{n}}
  L_n^{(\alpha)}\!\left(-2\epsilon(x-a)\right),\\[7pt]
P_n^{(IIIa)}(x)&=P_n^{(IIIb)}(x)
= \left(\frac{a-b}{2}\right)^{n}
  P_n^{(\alpha,\beta)}\!\left(\tfrac{2}{a-b}
  \left(x-\tfrac{a+b}{2}\right)\right),\\[7pt]
P_n^{(IIIc)}(x)
&= \left(\frac{2}{\beta}\right)^{-n}
  B_n^{(\alpha)}\!\left(\tfrac{2}{\beta} (x-a)\right), \\[7pt]
P_n^{(IIId)}(x)
&=\left(\frac{i}{\zeta}\right)^{-n}
  P_n^{\left(\epsilon-1+i\tfrac{\gamma-2 f \epsilon}{2 \zeta},
  \epsilon-1-i\dps\tfrac{\gamma-2 f \epsilon}{2 \zeta}\right)}
  \!\left(\tfrac{i}{\zeta} (x+f)\right),
\end{align*}
subject to the parameter restrictions listed in Table~\ref{table2}, where
$H_n$, $L_n^{(\alpha)}$, $B_n^{(\alpha)}$, and $P_n^{(\alpha,\beta)}$ denote
the polynomial families defined in Table~\ref{table}. The finite families described in Table \ref{table2} are often presented in the literature as having been newly discovered by Romanovski in his two-page note~\cite{R29}, which appeared in the same year as Bochner’s work. Romanovski’s genuine contribution should instead be understood as the identification of an associated weight function, a result of independent and lasting interest, yet one that does not amount to the introduction of new families of orthogonal polynomials. In fact, when the issue is the classification of polynomials, it is striking that \( P_n^{({IIIa})} \) and \( P_n^{({IIIb})} \), which are visibly the same polynomial, are treated as distinct objects merely because, depending on the domain of their parameters, they can be realised as OPRL with respect to different weights. In any event, the approach outlined above is by no means restricted to
\cite{KLS10} or to the \emph{NIST DLMF}, but is instead adopted throughout
the majority of the literature on the subject.

\begin{table}[h!]
\centering
\begin{tabular}{>{\columncolor{lightgray}}l|l|l}
\rowcolor{lightgray}
Family
& Monic
& Conditions \\ \hline
\rule{0pt}{20pt}Hermite      
& $P_n^{(I)}$      
& $\epsilon<0,\, \gamma\in \mathbb{R}$ 
\\
\rule{0pt}{20pt}Laguerre    
& $P_n^{(II)}$     
& $\alpha>-1,\, \epsilon<0,\, a\in \mathbb{R}$
\\
\rule{0pt}{20pt}Jacobi  
& $P_n^{(IIIa)}$   
& $\alpha, \beta>-1,\, a, b\in\mathbb{R},\, a<b$
\\
\rule{0pt}{20pt}Finite Jacobi  
& $P_n^{(IIIb)}$   
& $\alpha<-2N,\, \beta>-1,\, a, b\in\mathbb{R},\, a<b$
\\
\rule{0pt}{20pt}Finite Bessel  
& $P_n^{(IIIc)}$   
& $\alpha<-2N-1,\, \beta>0,\, a\in\mathbb{R}$
\\
\rule{0pt}{20pt}Finite pseudo-Jacobi 
& $P_n^{(IIId)}$   
& $\zeta>0,\, \epsilon<-\tfrac12,\, \gamma, f\in\mathbb{R}$
\\
\end{tabular}
\caption{Classical orthogonal polynomials according to the {\em NIST DLMF} and Koekoek--Lesky--Swarttouw classifications.}
\label{table2}
\end{table}

Well before Maroni’s works, and outside the specific context of orthogonal polynomials, the functional perspective had already been firmly advocated through concrete problems. An illustrative example in connection with the objectives of the present note
appears in a paper by Rota~\cite{R64}, where a particularly useful expression for the
number of partitions of a set with \( n \) elements is presented:
\begin{align}\label{bell}
B_n =\mathbf{b}(s^n),
\end{align}
where \(\mathbf{b}\) is a linear functional on \(\mathbb{C}[s]\), defined by prescribing its values on the (Hamel) basis \(\big((s)_n\big)_{n=0}^{\infty}\), namely,
\[
\mathbf{b}((s)_n)= 1,
\]
with \((s)_n \) denoting the falling factorial. In his work, Rota derived the known properties of \(B_n\) from equation \eqref{bell} with the same effortless clarity that Maroni would later bring to the derivation of many properties of the solutions to Bochner’s problem from equation \eqref{pearson}. The starting point of Rota was an auxiliary set \( S \) with  s elements, and the analysis of the structure of the set \( S^U \), comprising all functions from a set \( U \) of \( n \) elements into \( S \). Rota showed that the functional $\mathbf{b}$ is given by
\[
\mathbf{b}((s)_n)=\frac{1}{e}\,\sum_{k=0}^\infty \frac{(k)_n}{k!},
\]
and is therefore positive definite. This functional is related to what is commonly referred to in the literature as discrete orthogonal polynomials on linear (uniform) lattices, that is, polynomials orthogonal with respect to a positive discrete measure supported on a set of equally spaced points. In fact, Rota established a connection with the Charlier polynomials, which are mentioned by Szegő in his monograph and which belong to what are now regarded as a classical orthogonal polynomials on linear lattices. At the end of his work, Rota himself wrote: ``By the systematic use of linear functionals we can give a rigorous foundation to this calculus, as well as extend its uses in some directions.'' At this stage, the limited reception of Maroni’s ideas is particularly noteworthy, even though his approach fits naturally into a broader current of ideas stemming from the theory of LCS that permeated analysis at the time. This is all the more puzzling given that his works appeared during a period characterised by sustained efforts to broaden the notion of classical orthogonal polynomials.

In view of the foregoing, and without losing sight of Rota’s work, one is naturally led to ask who, precisely, are to be understood as the so-called classical orthogonal polynomials on linear lattices. Are these to be 
identified with the Meixner, Krawtchouk, Hahn, and Charlier families, as one might be led 
to infer from a substantial portion of the existing literature 
(see, for instance, \cite{L41, L62, NS86, NU88, NSU91, NU93, ARS95, GMS95, A06, AA07} 
and the references therein), including the \emph{NIST DLMF}? Or should they be identified with those presented in Theorems~5.1, 5.2, and~6.1 of \cite{KLS10}, the same reference we examined in our discussion of the Bochner case? What, then, is to be said of the so-called para-Krawtchouk polynomials (see \cite{VZ12}), which have recently emerged in applied contexts and whose discoverers themselves noted that these polynomials exhibit properties typical of families traditionally classified as classical? In light of the preceding discussion, it should not surprise the
reader if we now anticipate that the para-Krawtchouk families are simply special cases
emerging within a known family: not the Krawtchouk
polynomials, but rather the Hahn polynomials. Once again, the difficulty lies in a restrictive conception of orthogonality, which has hindered a proper formulation of the general problem and has given rise to artificial constructs such as bi-lattices and related notions. Clarifying this point constitutes one of the main objectives of the present work. Building upon
the solid foundations of duality theory in LCS, we will show, using equivalence classes,
which leave no room for evasion, which families truly deserve to be called the classical
discrete orthogonal polynomials on linear lattices. And it will come as no surprise that
we recover everything already found in the literature, from the well-known to the most
recent.

It is important to warn the reader that, although functional approaches do appear in the context of orthogonal polynomials on linear lattices, the literature contains serious conceptual flaws that have led to incorrect conclusions being accepted and repeatedly propagated as definitive results. In particular, certain assertions have entered the standard narrative not through rigorous validation, but through uncritical repetition, thereby obscuring the actual structure of the problem.
 A representative example can be found in~\cite{GMS95}, where the authors introduce a class $\mathcal{D}$ of 
regular linear functionals satisfying the equation~\eqref{pearson}, with the 
forward difference operator $\mathbf{\Delta}$ replacing the derivative $\mathbf{D}$:
\begin{align}\label{disc}
\mathbf{\Delta}(\phi \mathbf{a}) = \psi \mathbf{a}.
\end{align}
They then state, at the beginning of Section 4 (p. 160), that, up to translations and homotheties, the only regular linear functionals satisfying such an equation are those associated with the Meixner, Krawtchouk, Hahn, and Charlier families, writing: ``Using the known properties of the classical orthogonal polynomials, it is clear that their
functionals belong to $\mathcal{D}$. From Section 3 we can deduce that they are the only ones.'' A careful reading of \cite{GMS95}, however, shows that Section~3 does not actually contain arguments supporting this conclusion. Section~4, upon closer inspection, consists essentially of a table listing four families, of which, as we shall see, these are not the only ones, and three are essentially the same, whose parameters are moreover a priori restricted to real values. This restriction is not intrinsic to the problem: while the parameters may, in general, take complex values, genuine constraints arise on discrete subsets of the real line where the associated functionals fail to be regular, thereby collapsing any argument that relies on orthogonality.

In response to the question motivating the title of this work, the answer is no: this is not a further characterisation that artificially introduces new and ad hoc families, but rather a general and rigorous exposition of the solutions to an already existing problem. We have not sought to alter the framework, but instead to bring it to completion in a transparent and coherent manner, using a modern language that deliberately avoids obscuring the underlying structure behind excessive computational detail or unwieldy expressions.

In the following section, we fix the topological framework and collect the basic results required for the subsequent analysis. In Section~\ref{bo}, we formulate a Bochner-type problem in dual terms and study its functional-analytic content. In Section~\ref{canonical}, we introduce the canonical families as equivalence classes, which provides a structural identification of the families appearing in the literature and a precise description of their domains of regularity, thereby removing the intrinsic limitations of the positive definite setting. In Section~\ref{pd}, we treat the positive definite case in detail. In Section~\ref{01}, we recover Maroni’s setting as a limit process within the weak topology of the continuous dual. Throughout the paper, selected examples are used to situate earlier contributions within this framework and to illustrate the scope of the general theory. Finally, in Section~\ref{nist}, we express the families classically described in the literature in terms of our canonical formulation, making the correspondence explicit and highlighting the structural weaknesses inherent in attempts to characterize orthogonal polynomial families from purely algebraic properties within a Hilbertian framework. In Section~\ref{representacion} we make clear that our approach offers a clear and transparent procedure for obtaining the functional’s representation.

\section{Topological setting and preliminary results}
The material in this section is drawn primarily from \cite{M84, M88, M91a, M91b, CP25}, with adjustments made to render the present work self-contained. For consistency with Section \ref{sec1}, where the polynomial algebra is denoted \(\mathbb{C}[x]\) without any topological structure, we shall write \(\mathcal{P}\) for the same algebra endowed with the strict inductive-limit (LF) topology. Thus let
\[
\mathcal{P}=\bigcup_{n=0}^\infty\mathcal{P}_n,
\]
where \(\mathcal{P}_n\) is the \((n+1)\)-dimensional subspace of polynomials of degree at most \(n\), each equipped with its Euclidean locally convex topology, and equip \(\mathcal{P}\) with \(\varinjlim \mathcal{P}_n\). With this choice the canonical injections
\(i_n:\mathcal{P}_n\hookrightarrow\mathcal{P}\) are continuous and the links 
\[
\mathcal{P}_n\hookrightarrow\mathcal{P}_{n+1}
\]
are topological embeddings with closed range; hence \(\mathcal{P}\) is an LF-space, Hausdorff, complete, barrelled and bornological. In particular, a subset \(\mathcal{B}\subset\mathcal{P}\) is bounded precisely when it is contained and bounded in some finite step \(\mathcal{P}_m\). Denote by \(\mathcal{P}'\) the space of continuous linear functionals on \(\mathcal{P}\), paired with \(\mathcal{P}\) via 
\[
\langle\cdot,\cdot\rangle:\mathcal{P}'\times\mathcal{P}\to\mathbb{C}.
\]
Because the inductive system is strict with finite-dimensional steps, restriction identifies \(\mathcal{P}'\) with the projective limit \(\varprojlim \mathcal{P}_n'\). Unless explicitly stated otherwise, \(\mathcal{P}'\) carries the weak topology \(\sigma(\mathcal{P}',\mathcal{P})\); the strong dual topology \(\beta(\mathcal{P}',\mathcal{P})\) is the topology of uniform convergence on bounded subsets of \(\mathcal{P}\) and is generated by the seminorms 
\[
\|\mathbf{u}\|_{\mathcal{B}}=\sup_{p\in \mathcal{B}}|\langle \mathbf{u},p\rangle|
\]
for bounded \(\mathcal{B}\subset\mathcal{P}\), hence \(\mathcal{B}\subset\mathcal{P}_m\) for some \(m\). A linear map \(T:\mathcal{P}\to\mathcal{P}\) is continuous precisely when, for each \(n\), there exist \(m\) and 
\(C\ge 0\) with \(T(\mathcal{P}_n)\subseteq\mathcal{P}_m\) and 
\[
\|Tp\|_{\mathcal{P}_m}\le C\,\|p\|_{\mathcal{P}_n}, \quad p\in\mathcal{P}_n. 
\]
This covers operators that preserve degree or change it by a uniformly bounded amount (including affine changes of variable and the degree-lowering algebraic derivative), and consequently each such operator admits a well-defined, \(\sigma(\mathcal{P}',\mathcal{P})\)-continuous transpose \(T':\mathcal{P}'\to\mathcal{P}'\) determined by 
\[
\langle T'\mathbf{u},p\rangle=\langle \mathbf{u},Tp\rangle;
\] functoriality holds in the form \((ST)'=T'S'\) and \((\mathrm{id}_{\mathcal{P}})'=\mathrm{id}_{\mathcal{P}'}\). In particular, for fixed \(p_0\in\mathcal{P}\) the multiplication operator \[M_{p_0}:q\mapsto p_0\,q\] is continuous since \(M_{p_0}(\mathcal{P}_n)\subseteq\mathcal{P}_{n+\deg p_0},\) and its transpose, denoted \(p_0\mathbf{u}\), satisfies 
\[\langle p_0\mathbf{u},q\rangle=\langle \mathbf{u},p_0\,q\rangle, \quad q\in\mathcal{P}.
\] 
Henceforth, every expression in \(\mathcal{P}'\) obtained from finitely many transposes \(T'\) of continuous \(T\) and from transposed multiplications \(M_p'\) is a continuous linear functional on \(\mathcal{P}\); an identity \[\mathbf{v}=\mathbf{w}\] in \(\mathcal{P}'\) is to be understood  (see, for instance, \eqref{pearson}, \eqref{disc}) in the dual pair \((\mathcal{P}',\mathcal{P})\), i.e.
\[
\langle \mathbf{v},q\rangle=\langle \mathbf{w},q\rangle, \quad q\in\mathcal{P},
\]
which, by separation of points under the pairing, is equivalent to equality in \(\mathcal{P}'\) (for both the weak and the strong dual topologies). Since the class of continuous linear maps on \(\mathcal{P}\) is stable under composition and contains all operators that preserve degree or change it by a uniformly bounded amount, the transposes of the operators introduced later (in particular, translations, homotheties, and the discrete derivative/average on linear lattices) are automatically well defined and \(\sigma(\mathcal{P}',\mathcal{P})\)-continuous. All subsequent duality statements are therefore valid as identities of continuous linear functionals in this precise sense.


\begin{definition}\cite[Definition~6.1]{CP25}
Let $\beta\in\mathbb{C}$.
Define the translation operator
\[
\tau_\beta : \mathcal{P} \to \mathcal{P}, \quad
(\tau_\beta p)(x) = p(x-\beta), \quad p\in\mathcal{P}.
\]
The transpose of $\tau_{-\beta}$ is denoted by
\[
\boldsymbol{\tau}_\beta : \mathcal{P}' \to \mathcal{P}',
\]
and is defined by
\[
\langle \boldsymbol{\tau}_\beta \mathbf{u},\, p\rangle
=
\langle \mathbf{u},\, \tau_{-\beta}(p)\rangle,
\quad p\in\mathcal{P},
\]
for all $\mathbf{u}\in\mathcal{P}'$.
\end{definition}

\begin{definition}\cite[Definition~6.2]{CP25}
Let $\alpha\in\mathbb{C}^{\times}=\mathbb{C}\setminus\{0\}$.
Define the homothetic operator
\[
h_\alpha : \mathcal{P} \to \mathcal{P}, \quad
(h_\alpha p)(x) = p(\alpha x), \quad p\in\mathcal{P}.
\]
The transpose of $h_\alpha$ is denoted by
\[
\boldsymbol{h}_\alpha : \mathcal{P}' \to \mathcal{P}',
\]
and is defined by
\[
\langle \boldsymbol{h}_\alpha \mathbf{u},\, p\rangle
=
\langle \mathbf{u},\, h_\alpha(p)\rangle,
\quad p\in\mathcal{P},
\]
for all $\mathbf{u}\in\mathcal{P}'$.
\end{definition}

We fix the following convention.

\begin{definition}
Let $\alpha,\beta:\mathcal{P}\to\mathcal{P}$ be linear maps admitting transposes.
Set
\[
\boldsymbol{\alpha}=\alpha', \,
\boldsymbol{\beta}=\beta' : \mathcal{P}'\to\mathcal{P}'.
\]
The composition of $\boldsymbol{\alpha}$ and $\boldsymbol{\beta}$ is the linear map
\[
\boldsymbol{\alpha}\circ\boldsymbol{\beta}:\mathcal{P}'\to\mathcal{P}'
\]
defined by
\[
\langle (\boldsymbol{\alpha}\circ\boldsymbol{\beta})\,\mathbf{u},\,p\rangle
=
\langle \mathbf{u},\,(\beta\circ\alpha)(p)\rangle,
\quad p\in\mathcal{P},
\]
for all $\mathbf{u}\in\mathcal{P}'$.
\end{definition}

The following proposition records elementary properties of the translation and
homothetic operators and their transposes.

\begin{proposition}
\label{proposition:lineartransproposition}
Let $\alpha \in\mathbb{C}^{\times}$, $\beta \in \mathbb{C}$, $\mathbf{u}\in\mathcal{P}'$, and
$p,q\in\mathcal{P}$. The following identities hold.
\begin{multicols}{2}
\begin{enumerate}[label=\textnormal{(\alph*)},
                  leftmargin=*,
                  itemsep=1.5ex,
                  parsep=0.7ex,
                  topsep=1.5ex]

\item\label{proposition:lineartransproposition-b}
$\left(\tau_\beta\circ\tau_{-\beta}\right)\!\left(p\right)
= \left(\tau_{-\beta}\circ\tau_\beta\right)\!\left(p\right) = p$.

\item\label{proposition:lineartransproposition-c}
$\left(h_\alpha\circ h_{\alpha^{-1}}\right)\!\left(p\right)
= \left(h_{\alpha^{-1}}\circ h_\alpha\right)\!\left(p\right) = p$.

\item\label{proposition:lineartransproposition-f}
$\tau_\beta\!\left(pq\right) = \tau_\beta\!\left(p\right)\,\tau_\beta\!\left(q\right)$.

\item\label{proposition:lineartransproposition-g}
$h_\alpha\!\left(pq\right) = h_\alpha\!\left(p\right)\,h_\alpha\!\left(q\right)$.

\item\label{proposition:lineartransproposition-m}
$\boldsymbol{\tau}_\beta\!\left(p\mathbf{u}\right)
= \tau_\beta\!\left(p\right)\,\boldsymbol{\tau}_\beta\,\mathbf{u}$.

\item\label{proposition:lineartransproposition-n}
$\boldsymbol{h}_\alpha\!\left(p\mathbf{u}\right)
= h_{\alpha^{-1}}\!\left(p\right)\,\boldsymbol{h}_\alpha\mathbf{u}$.

\item\label{proposition:lineartransproposition-aux1}
$(\boldsymbol{h}_\beta \circ \boldsymbol{\tau}_\alpha)\,\mathbf{u}= (\boldsymbol{\tau}_{\alpha \beta}\circ  \boldsymbol{h}_\beta)\,\mathbf{u}$

\item\label{proposition:lineartransproposition-aux2}
$(\boldsymbol{\tau}_\alpha \circ \boldsymbol{\tau}_{-\alpha})\, \mathbf{u}= (\boldsymbol{\tau}_{-\alpha}\circ  \boldsymbol{\tau}_\beta)\,\mathbf{u}$

\item\label{proposition:lineartransproposition-aux3}
$(\boldsymbol{h}_\beta \circ \boldsymbol{h}_{\beta^{-1}})\, \mathbf{u}=(\boldsymbol{h}_{\beta^{-1}} \circ  \boldsymbol{h}_\beta)\,\mathbf{u}$

\end{enumerate}
\end{multicols}
\end{proposition}

\begin{proof}
Properties~\ref{proposition:lineartransproposition-b}-\ref{proposition:lineartransproposition-g}, \ref{proposition:lineartransproposition-aux1}, \ref{proposition:lineartransproposition-aux2}, and \ref{proposition:lineartransproposition-aux3}
are listed in~\cite[Proposition~6.1]{CP25}.
It remains to prove~\ref{proposition:lineartransproposition-m}
and~\ref{proposition:lineartransproposition-n}. 

Let $q\in\mathcal{P}$. By definition of $\boldsymbol{\tau}_\beta$,
\[
\langle \boldsymbol{\tau}_\beta(p\mathbf{u}),\,q\rangle
=
\langle \mathbf{u},\,p\,\tau_{-\beta}(q)\rangle.
\]
Using~\ref{proposition:lineartransproposition-f} with $\beta$ replaced by $-\beta$
and~\ref{proposition:lineartransproposition-b}, one obtains
\[
p\,\tau_{-\beta}(q)
=
\tau_{-\beta}\!\left(\tau_\beta(p)\,q\right).
\]
Hence,
\begin{align*}
\langle \boldsymbol{\tau}_\beta(p\mathbf{u}),\,q\rangle
&=
\langle \mathbf{u},\,\tau_{-\beta}\!\left(\tau_\beta(p)\,q\right)\rangle
=
\langle \boldsymbol{\tau}_\beta\mathbf{u},\,\tau_\beta(p)\,q\rangle\\[7pt]
&=
\langle \tau_\beta(p)\,\boldsymbol{\tau}_\beta\mathbf{u},\,q\rangle,
\end{align*}
which proves~\ref{proposition:lineartransproposition-m}.

By definition of $\boldsymbol{h}_\alpha$,
\[
\langle \boldsymbol{h}_\alpha(p\mathbf{u}),\,q\rangle
=
\langle \mathbf{u},\,p\,h_\alpha(q)\rangle.
\]
Using~\ref{proposition:lineartransproposition-g}
and~\ref{proposition:lineartransproposition-c}, one obtains
\[
p\,h_\alpha(q)
=
h_\alpha\!\left(h_{\alpha^{-1}}(p)\,q\right).
\]
Hence,
\begin{align*}
\langle \boldsymbol{h}_\alpha(p\mathbf{u}),\,q\rangle
&=
\langle \mathbf{u},\,h_\alpha\!\left(h_{\alpha^{-1}}(p)\,q\right)\rangle
=
\langle \boldsymbol{h}_\alpha\mathbf{u},\,h_{\alpha^{-1}}(p)\,q\rangle\\[7pt]
&=
\langle h_{\alpha^{-1}}(p)\,\boldsymbol{h}_\alpha\mathbf{u},\,q\rangle,
\end{align*}
which proves~\ref{proposition:lineartransproposition-n}.
\end{proof}

We define the following equivalence relation on $\mathcal{P}'$.

\begin{definition}\cite[p.~114]{CP25}\label{defequi}
Let $\mathbf{u},\mathbf{v}\in\mathcal{P}'$.
The relation $\sim^*$ is defined by
\[
\mathbf{u}\sim^*\mathbf{v}
\]
if there exist $\beta\in\mathbb{C}$ and $\alpha\in\mathbb{C}^{\times}$ such that
\[
\mathbf{v}
=
\left(\boldsymbol{\tau}_{\beta}\circ \boldsymbol{h}_{\alpha^{-1}}\right)\mathbf{u}.
\]
\end{definition}

The next proposition records the transformation of orthogonal polynomial sequences induced by \(\sim^*\).

\begin{proposition}\cite[Theorem~6.4]{CP25}
\label{proposition:affineOPS}
Let $\mathbf{u},\mathbf{v}\in\mathcal{P}'$ be nonzero, and assume that
\(
\mathbf{u}\sim^{*}\mathbf{v},
\)
with $\mathbf{u}$ regular.
Let $(P_n)_{n\geq 0}$ denote the monic orthogonal polynomial sequence with respect to
$\mathbf{u}$, satisfying the recurrence relation
\[
xP_n(x)=P_{n+1}(x)+a_nP_n(x)+b_nP_{n-1}(x),
\]
with $P_{-1}=0$ and $P_0=1$. Then $\mathbf{v}$ is regular, and the sequence
$(Q_n)_{n\geq 0}$ defined by
\[
Q_n
=
\frac{1}{\alpha}\,\left(\tau_{\beta} \circ h_{\alpha}\right)\!\left(P_n\right)
\]
is the monic orthogonal polynomial sequence with respect to $\mathbf{v}$.
Moreover, $(Q_n)_{n\geq 0}$ satisfies the recurrence relation
\[
xQ_n(x)=Q_{n+1}(x)+c_nQ_n(x)+d_nQ_{n-1}(x),
\]
with $Q_{-1}=0$ and $Q_0=1$, where
\[
c_n=\frac{a_n}{\alpha}+\beta,
\quad
d_n=\frac{b_n}{\alpha^{2}}.
\]
\end{proposition}

The following operators provide a convenient formulation of classicality in the
present setting.

\begin{definition}\cite[Definition~1.2]{CP25}
\label{def:D-linear}
Let ${X}:\mathbb{C}\to\mathbb{C}$ be the linear lattice
\({X}(s)=\mathfrak{c}\,s+\mathfrak{d}\) with $\mathfrak{c}\neq 0$. Define the ${X}$-derivative as the linear operator
\[
\mathrm{D}_{{X}}:\mathcal{P}\to\mathcal{P}
\]
given, for $x={X}(s)$, by
\[
(\mathrm{D}_{{X}}p)(x)
=
\frac{p\left(x+\dps \frac{\mathfrak{c}}{2}\right)-p\left(x-\dps\frac{\mathfrak{c}}{2}\right)}{\mathfrak{c}},
\quad p\in\mathcal{P}.
\]
Whenever the lattice is written explicitly as ${X}(s)=\mathfrak{c}s+\mathfrak{d}$,
we also write
\(
\mathrm{D}_{\mathfrak{c}s+\mathfrak{d}}
\)
to denote $\mathrm{D}_{{X}}$.
The transpose of $\mathrm{D}_{{X}}$ is denoted by
\[
\boldsymbol{\mathrm{D}}_{{X}}:\mathcal{P}'\to\mathcal{P}',
\]
and is defined by
\[
\langle \boldsymbol{\mathrm{D}}_{{X}}\mathbf{u},\,p\rangle
=
-\langle \mathbf{u},\,\mathrm{D}_{{X}}(p)\rangle,
\quad p\in\mathcal{P},
\]
for all $\mathbf{u}\in\mathcal{P}'$.
\end{definition}

\begin{definition}\cite[Definition~1.3]{CP25}
\label{def:S-linear}
Let ${X}:\mathbb{C}\to\mathbb{C}$ be the linear lattice
\(
{X}(s)=\mathfrak{c}\,s+\mathfrak{d}\) with $\mathfrak{c}\neq 0$. Define the ${X}$-average as the linear operator
\[
\mathrm{S}_{{X}}:\mathcal{P}\to\mathcal{P}
\]
given, for $x={X}(s)$, by
\[
(\mathrm{S}_{{X}}p)(x)
=
\frac{
p\!\left(x+\dps\frac{\mathfrak{c}}{2}\right)
+
p\!\left(x-\dps\frac{\mathfrak{c}}{2}\right)
}{2},
\quad p\in\mathcal{P}.
\]
Whenever the lattice is written explicitly as ${X}(s)=\mathfrak{c}s+\mathfrak{d}$,
we also write
\(
\mathrm{S}_{\mathfrak{c}s+\mathfrak{d}}
\)
to denote $\mathrm{S}_{{X}}$.
The transpose of $\mathrm{S}_{{X}}$ is denoted by
\[
\boldsymbol{\mathrm{S}}_{{X}}:\mathcal{P}'\to\mathcal{P}',
\]
and is defined by
\[
\langle \boldsymbol{\mathrm{S}}_{{X}}\mathbf{u},\,p\rangle
=
\langle \mathbf{u},\,\mathrm{S}_{{X}}(p)\rangle,
\quad p\in\mathcal{P},
\]
for all $\mathbf{u}\in\mathcal{P}'$.
\end{definition}

\begin{remark}\label{rem:DX-SX-limit}
The symmetric shifts $x\pm\mathfrak c/2$ in the definitions of $\mathrm D_X$ and
$\mathrm S_X$ are chosen so as to recover the correct continuous limit as
$\mathfrak c\to0$. Indeed, for any $p\in\mathcal P$ the Taylor expansions at $x$
are finite and yield
\[
p\!\left(x\pm\frac{\mathfrak c}{2}\right)
=
\sum_{k=0}^\infty\frac{(\pm\mathfrak c/2)^k}{k!}\,p^{(k)}(x).
\]
By symmetry, all even $($respectively, odd$)$ terms cancel in the definition of
$\mathrm D_X$ $($respectively, $\mathrm S_X$$)$, and one obtains
\[
(\mathrm D_X p)(x)
=
\sum_{j=0}^\infty\frac{(\mathfrak c/2)^{2j}}{(2j+1)!}\,p^{(2j+1)}(x),
\quad
(\mathrm S_X p)(x)
=
\sum_{j=0}^\infty\frac{(\mathfrak c/2)^{2j}}{(2j)!}\,p^{(2j)}(x),
\]
where the sums are finite. In particular,
\[
\mathrm D_X p \xrightarrow[\mathfrak c\to0]{} p',
\quad
\mathrm S_X p \xrightarrow[\mathfrak c\to0]{} p,
\]
in the finite-dimensional space $\mathcal P_{\deg p}$, and hence in $\mathcal P$ endowed with its LF-topology.
By transposition, for every $\mathbf u\in\mathcal P'$,
\[
\boldsymbol{\mathrm D}_X\mathbf u \xrightarrow[\mathfrak c\to0]{} \mathbf D\,\mathbf u,
\quad
\boldsymbol{\mathrm S}_X\mathbf u \xrightarrow[\mathfrak c\to0]{} \mathbf u,
\]
in $\sigma(\mathcal P',\mathcal P)$, where $\mathbf D$ denotes the transpose of the usual derivative on $\mathcal P$.
\end{remark}

The following results provide necessary and sufficient conditions for
regularity in the present setting.

\begin{theorem}\cite[Theorem~9.3]{CP25}
\label{thm:linear-lattice}
Let ${X}:\mathbb{C}\to\mathbb{C}$ be the linear lattice
\({X}(s)=\mathfrak{c}\,s+\mathfrak{d}\) with \(\mathfrak{c}\neq 0\). Let $\mathbf{u}_{\mathfrak{c}}\in\mathcal{P}'$ be nonzero and assume that there exist polynomials
$\phi$ and $\psi$, of degrees at most two and one, respectively, not both identically zero, such that
\[
\boldsymbol{\mathrm{D}}_{{X}}(\phi\,\mathbf{u}_\mathfrak{c})
=
\boldsymbol{\mathrm{S}}_{{X}}(\psi\,\mathbf{u}_\mathfrak{c}),
\]
where $\mathrm{D}_{{X}}$ and $\mathrm{S}_{{X}}$ are the operators of
Definitions~\ref{def:D-linear}--\ref{def:S-linear}. Write
\(
\phi(x)=a x^{2}+b x+c\), $\psi(x)=d x+e$,
and set
\[
d_n = a n + d,
\quad
e_n = b n + e,
\quad
\phi^{[n]}(x)
=\phi(x)+\frac{1}{4}\,n d_n\mathfrak{c}^{2}.
\]
Then $\mathbf{u}$ is regular if and only if
\[
d_n\neq 0, \quad
\phi^{[n]}\!\left(-\frac{e_n}{d_{2n}}\right)\neq 0.
\]
When these conditions hold for all $n\in\mathbb{N}$, there exists a monic orthogonal
polynomial sequence $(P_n)_{n\ge0}$ satisfying
\[
P_{n+1}(x)
=
(x-a_n)\,P_n(x)-b_n^{(\mathfrak{c})}\,P_{n-1}(x),
\]
with $P_{-1}=0$ and $P_0=1$.
If the conditions hold only for $n\le N<\infty$, then the above recurrence defines a
finite monic orthogonal polynomial sequence $(P_n)_{n=0}^N$. In either case, the recurrence coefficients are given by
\begin{align*}
a_n
&=
\frac{n\,e_{n-1}}{d_{2n-2}}
-
\frac{(n+1)\,e_n}{d_{2n}},\quad n\in \mathbb{N}\\[7pt]
b_n^{(\mathfrak{c})}
&=
-\frac{n\,d_{n-2}}{d_{2n-3}\,d_{2n-1}}\,
\phi^{[\,n-1\,]}\!\left(-\frac{e_{n-1}}{d_{2n-2}}\right),\quad n\in \mathbb{N}^\times.
\end{align*}
\end{theorem}

\section{The functional Bochner-type equation}\label{bo}
We now introduce the notion of a classical linear functional in the present framework.

\begin{definition}[$1$-classicality]\label{defclass}
Let $X:\mathbb{C}\to\mathbb{C}$ be a linear lattice of the form
$X(s)=\mathfrak c s+\mathfrak d$ with $\mathfrak c\neq0$.
A nonzero functional $\mathbf u_{\mathfrak c}\in\mathcal P'$ is said to be
\emph{$1$-classical} if there exist polynomials $\phi$
and $\psi$ of degrees at most two and one, respectively, not both identically
zero, such that
\[
\boldsymbol{\mathrm D}_X(\phi\,\mathbf u_{\mathfrak c})
=
\boldsymbol{\mathrm S}_X(\psi\,\mathbf u_{\mathfrak c}).
\]
\end{definition}

At this stage it becomes clear that the theory unfolds under the guiding action of translations and dilations on the dual space. The following proposition brings this structure into sharp focus.

\begin{proposition}
\label{proposition:operatorshape}
Let ${X}:\mathbb C\to\mathbb C$ be the linear lattice
\(
{X}(s)=\mathfrak c\,s+\mathfrak d
\)
with $\mathfrak c\neq0$. Then the operators $\mathrm D_{X}$ and $\mathrm S_{X}$ defined in Definitions~\ref{def:D-linear}--\ref{def:S-linear} admit the representations
\[
\mathrm D_{X}
=
\frac{1}{\mathfrak c}\bigl(\tau_{-\mathfrak c/2}-\tau_{\mathfrak c/2}\bigr),
\quad
\mathrm S_{X}
=
\frac12\bigl(\tau_{-\mathfrak c/2}+\tau_{\mathfrak c/2}\bigr).
\]
Moreover, the corresponding transposed operators
$\boldsymbol{\mathrm D}_{X}$ and $\boldsymbol{\mathrm S}_{X}$
satisfy
\[
\boldsymbol{\mathrm D}_{X}
=
\frac{1}{\mathfrak c}\bigl(\boldsymbol{\tau}_{-\mathfrak c/2}-\boldsymbol{\tau}_{\mathfrak c/2}\bigr),
\quad
\boldsymbol{\mathrm S}_{X}
=
\frac12\bigl(\boldsymbol{\tau}_{-\mathfrak c/2}+\boldsymbol{\tau}_{\mathfrak c/2}\bigr).
\]
In particular, $\mathrm D_{X}$ and $\mathrm S_{X}$ depend only on the slope
$\mathfrak c$ of the lattice and not on the intercept $\mathfrak d$. Hence, for every
$\mathfrak e\in\mathbb C$,
\[
\mathrm D_{\mathfrak c s+\mathfrak d}
=
\mathrm D_{\mathfrak c s+(\mathfrak d+\mathfrak e)},
\quad
\mathrm S_{\mathfrak c s+\mathfrak d}
=
\mathrm S_{\mathfrak c s+(\mathfrak d+\mathfrak e)},
\]
with analogous identities for the transposed operators. Moreover, both
$\mathrm D_{X}$ and $\mathrm S_{X}$ are invariant under the transformation
$\mathfrak c\mapsto-\mathfrak c$.
\end{proposition}

\begin{proof}
For $x={X}(s)$ one has
\(
x\pm \mathfrak c/2={X}(s\pm\tfrac12)
\).
Therefore, for every $p\in\mathcal P$,
\[
(\mathrm D_{X}p)(x)
=
\frac{p(x+\mathfrak c/2)-p(x-\mathfrak c/2)}{\mathfrak c}
=
\frac{(\tau_{-\mathfrak c/2}p)(x)-(\tau_{\mathfrak c/2}p)(x)}{\mathfrak c},
\]
and similarly
\[
(\mathrm S_{X}p)(x)
=
\frac{(\tau_{-\mathfrak c/2}p)(x)+(\tau_{\mathfrak c/2}p)(x)}{2}.
\]
This yields the stated representations. The formulas for the transposed
operators follow directly from the definitions of $\boldsymbol{\mathrm D}_{X}$
and $\boldsymbol{\mathrm S}_{X}$ by duality. The dependence on $\mathfrak c$ alone is immediate from the above expressions,
and invariance under $\mathfrak c\mapsto-\mathfrak c$ follows from their symmetry
under interchange of $\tau_{\mathfrak c/2}$ and $\tau_{-\mathfrak c/2}$.
\end{proof}

\begin{proposition}\label{proposition:affine-invariance}
Let $X:\mathbb{C}\to\mathbb{C}$ be the linear lattice
$X(s)=\mathfrak c\,s+\mathfrak d$ with $\mathfrak c\neq0$, and let
$\mathbf u_{\mathfrak c}\in\mathcal P'$ and $\phi,\psi\in\mathcal P$ satisfy
\[
\boldsymbol{\mathrm D}_{X}(\phi\,\mathbf u_{\mathfrak c})
=
\boldsymbol{\mathrm S}_{X}(\psi\,\mathbf u_{\mathfrak c}).
\]
Fix $\beta\in\mathbb C^\times$ and $\alpha\in\mathbb C$, and set
\[
\mathbf v_{\beta\mathfrak c}
=
(\boldsymbol{\tau}_{\alpha}\circ \boldsymbol{h}_{\beta})\,\mathbf u_{\mathfrak c},
\quad
X_{\beta}(s)=(\beta X)(s)=\beta\mathfrak c\,s+\beta\mathfrak d,
\]
together with
\[
\phi_{\beta,\alpha}
=
\beta\,\tau_{\alpha}\!\bigl(h_{\beta^{-1}}(\phi)\bigr),
\quad
\psi_{\beta,\alpha}
=
\tau_{\alpha}\!\bigl(h_{\beta^{-1}}(\psi)\bigr).
\]
Then $\mathbf v_{\beta\mathfrak c}$ satisfies
\[
\boldsymbol{\mathrm D}_{X_{\beta}}\!\bigl(\phi_{\beta,\alpha}\,\mathbf v_{\beta\mathfrak c}\bigr)
=
\boldsymbol{\mathrm S}_{X_{\beta}}\!\bigl(\psi_{\beta,\alpha}\,\mathbf v_{\beta\mathfrak c}\bigr).
\]
\end{proposition}

\begin{proof}
The argument rests on the natural behaviour of the operators
$\boldsymbol{\mathrm D}_X$ and $\boldsymbol{\mathrm S}_X$ under affine changes of
variables. Starting from the functional equation
\(
\boldsymbol{\mathrm D}_{X}(\phi\,\mathbf u_{\mathfrak c})
=
\boldsymbol{\mathrm S}_{X}(\psi\,\mathbf u_{\mathfrak c}),
\)
we apply successively the transpose homothety $\boldsymbol h_\beta$ and the
transpose translation $\boldsymbol{\tau}_\alpha$ (Proposition~\ref{proposition:lineartransproposition}\, \ref{proposition:lineartransproposition-aux2} and~\ref{proposition:lineartransproposition-aux3}). By
Proposition~\ref{proposition:operatorshape}, these transformations transport the
operators to the rescaled lattice $X_\beta$ (Proposition~\ref{proposition:lineartransproposition}\, \ref{proposition:lineartransproposition-aux1}), while the corresponding product
rules for transpose maps
(Proposition~\ref{proposition:lineartransproposition}\, \ref{proposition:lineartransproposition-m} and~\ref{proposition:lineartransproposition-n}) account
for the modification of the polynomial coefficients. This leads precisely to
the claimed identity.
\end{proof}

We insist on a fundamental point: the approaches found in the literature ultimately stem from the same underlying problem considered here. In particular, the problem initiated from the equation \eqref{disc} and the analysis based on the difference equation (see \cite[Chapters~5 and~6]{KLS10})
\begin{equation}\label{eq:KLSdiff-eq}
\bigl(a(x-1)^2+b(x-1)+c\bigr)\,(\Delta\nabla y)(x)
+
\bigl(d(x-1)+e\bigr)\,(\nabla y)(x)
=
\lambda_n\,y(x),
\end{equation}
are in fact concerned with the same. We write
\(
\Delta,\nabla:\mathcal P\to\mathcal P
\)
for the forward and backward difference operators defined by
\[
(\Delta p)(s)=p(s+1)-p(s),
\quad
(\nabla p)(s)=p(s)-p(s-1),
\quad p\in\mathcal P,
\]
and denote by
\(
\boldsymbol{\nabla}, \boldsymbol{\Delta}:\mathcal P'\to\mathcal P'
\)
their respective transposes. By definition of the transpose translation
operators, these satisfy the identities
\[
\boldsymbol{\nabla}
=
\boldsymbol{\tau}_{0}-\boldsymbol{\tau}_{1}, \quad
\boldsymbol{\Delta}
=
\boldsymbol{\tau}_{-1}-\boldsymbol{\tau}_{0},
\]
in \(\mathcal P'\).

\begin{obs}
\label{proposition:disc-to-centered}
Let ${X}_0:\mathbb{C}\to\mathbb{C}$ be the linear lattice
\(X_0(s)=s\). Let $\mathbf u\in\mathcal P'$ and let $\phi,\psi\in\mathcal P$.
Assume that $\mathbf u$ satisfies the functional equation
\[
\boldsymbol{\Delta}(\phi\,\mathbf u)=\psi\,\mathbf u.
\]
Then $\mathbf u$ also satisfies
\[
\boldsymbol{\mathrm D}_{X_0}\!\bigl((2\phi+\psi)\,\mathbf u\bigr)
=
2\,\boldsymbol{\mathrm S}_{X_0}(\psi\,\mathbf u).
\]
Conversely, if $\mathbf u$ is $1$-classical in the sense of
Definition~\ref{defclass} with respect to the lattice $X_0(s)=s$, then
\[
\boldsymbol{\Delta}\!\bigl((2\phi-\psi)\,\mathbf u\bigr)
=
2\,\psi\,\mathbf u.
\]
\end{obs}

\begin{proof}
Since \(X_0(s)=s\), Proposition~\ref{proposition:operatorshape} gives the
representations
\[
\boldsymbol{\mathrm D}_{X}
=
\boldsymbol{\tau}_{-1/2}-\boldsymbol{\tau}_{1/2},
\quad
2\,\boldsymbol{\mathrm S}_{X}
=
\boldsymbol{\tau}_{-1/2}+\boldsymbol{\tau}_{1/2}.
\]
Together with
\(
\boldsymbol{\Delta}=\boldsymbol{\tau}_{-1}-\boldsymbol{\tau}_{0},
\)
this yields the intertwining relations
\[
\boldsymbol{\tau}_{1/2}\circ\boldsymbol{\Delta}
=
\boldsymbol{\mathrm D}_{X},
\quad
2\,\boldsymbol{\mathrm S}_{X}-\boldsymbol{\mathrm D}_{X}
=
2\,\boldsymbol{\tau}_{1/2}.
\]
Applying \(\boldsymbol{\tau}_{1/2}\) to
\(\boldsymbol{\Delta}(\phi\,\mathbf{u})=\psi\,\mathbf{u}\)
and using the above identities gives the first claim.
The converse follows by reversing this argument and using the defining
equation of \(1\)-classicality.
\end{proof}

\begin{obs}
\label{prop:KLS10-functional}
Let ${X}_0:\mathbb{C}\to\mathbb{C}$ be the linear lattice
\(X_0(s)=s\). Let $\mathbf u\in\mathcal P'$ and consider the polynomials
\[
\phi(x)=a(x-1)^2+b(x-1)+c,
\quad
\psi(x)=\alpha(x-1)+\beta.
\]
Then a polynomial sequence satisfies the difference equation
\eqref{eq:KLSdiff-eq} if and only if its associated functional
\(\mathbf{u}\) satisfies
\begin{equation}
\label{eq:KLS10diff-eq}
\boldsymbol{\nabla}(\phi\,\mathbf{u})=\psi\,\mathbf{u}.
\end{equation}
Furthermore, under these assumptions, \(\mathbf{u}\) is \(1\)-classical with
respect to the lattice \(X_0(s)=s\). More precisely, it solves
\[
\boldsymbol{\mathrm D}_{X_0}\!\bigl((2\phi-\psi)\,\mathbf u\bigr)
=
2\,\boldsymbol{\mathrm S}_{X_0}(\psi\,\mathbf u).
\]
Conversely, whenever $\mathbf u$ is $1$-classical in the sense of
Definition~\ref{defclass} with respect to the lattice $X_0(s)=s$, it also satisfies
\[
\boldsymbol{\nabla}\!\bigl((2\phi+\psi)\,\mathbf u\bigr)
=
2\,\psi\,\mathbf u.
\]
\end{obs}

\begin{proof}
The equivalence between the difference equation \eqref{eq:KLSdiff-eq} and the
functional identity
\(\boldsymbol{\nabla}(\phi\,\mathbf{u})=\psi\,\mathbf{u}\)
follows from the definition of the transpose backward difference operator and
the standard duality between recurrence relations and functional equations.
The final statement follows by analogy with
Observation~\ref{proposition:disc-to-centered}, replacing the forward difference
operator $\boldsymbol{\Delta}$ by the backward difference operator
$\boldsymbol{\nabla}$.
\end{proof}

It is relatively common in the literature to encounter families of
orthogonal polynomials that appear, at first sight, to be new.
A recent example is provided by the so-called para-Krawtchouk polynomials
(see~\cite{VZ12}), briefly mentioned in the Introduction. We now turn to
this family solely to place it within our framework, although we shall
return to this example later on.

\begin{eje}[Para-Krawtchouk polynomials]
\label{ex:paraKrawtchouk-1classical-form}
Let $N\in\mathbb N^{\times}$ and $\gamma\in\mathbb C$. Assume that there exists a finite sequence of orthogonal polynomials $(p_n)_{n=0}^{N}$ satisfying, for every $n\in\{0,1,\dots,N\}$ and every $x\in\mathbb C$,
\begin{align}\label{paraK}
a(x)\,p_n(x+2)+b(x)\,p_n(x-2)
-\bigl(a(x)+b(x)\bigr)\,p_n(x)
=
2n(n-N)\,p_n(x),
\end{align}
where
\[
a(x)=\frac{(x-N+1)(x-N+1-\gamma)}{2},
\quad
b(x)=\frac{x(x-\gamma)}{2}.
\]
The finiteness of the sequence is already encoded in these coefficients. Indeed,
\[
b(0)=0,
\quad
a(N-1)=0.
\]
Set
\[
\phi(x)=a(x)+b(x),
\quad
\psi(x)=a(x)-b(x).
\]
A direct computation shows that $\phi$ is a nonzero polynomial of degree $2$ and
that $\psi$ is a polynomial of degree at most $1$ $($of degree exactly $1$
whenever $N\neq1$$)$. More explicitly,
\begin{align}
\label{aux1}\phi(x)
&=
x^{2}-(N-1+\gamma)\,x
+\frac{(N-1)(N-1+\gamma)}{2},\\[7pt]
\label{aux2}\psi(x)
&=
-(N-1)\,x+\frac{(N-1)(N-1+\gamma)}{2}.
\end{align}
Let $X(s)=2s$ be the linear lattice $($with $\mathfrak c=2$ and $\mathfrak d=0$$)$.
Then $x\pm2=X(s\pm1)$, and Proposition~\ref{proposition:operatorshape} yields
\[
\mathrm D_{X}
=
\frac{1}{2}\bigl(\tau_{-1}-\tau_{1}\bigr),
\quad
\mathrm S_{X}
=
\frac{1}{2}\bigl(\tau_{-1}+\tau_{1}\bigr),
\]
together with the corresponding transpose operators on $\mathcal P'$.
Let $\mathbf u\in\mathcal P'$ be the linear functional with respect to
which the sequence $(p_n)_{n=0}^{N}$ is orthogonal. Then the difference equation
\eqref{paraK} can be rewritten as the
functional identity
\[
\boldsymbol{\mathrm D}_{X}(\phi\,\mathbf u)
=
\boldsymbol{\mathrm S}_{X}(\psi\,\mathbf u),
\]
with $\phi$ and $\psi$ given by \eqref{aux1} and \eqref{aux2}. In particular, $\mathbf u$ is $1$-classical in the
sense of Definition~\ref{defclass} with respect to the lattice $X(s)=2s$. To interpret regularity in this setting, we appeal to the explicit recurrence
structure provided by Theorem~\ref{thm:linear-lattice}. Writing $\phi(x)=ax^{2}+bx+c$ and $\psi(x)=dx+e$, this corresponds to
\[
a=1,\quad
b=-(N-1+\gamma),\quad
c=\frac{(N-1)(N-1+\gamma)}{2},
\]
\[
d=-(N-1),\quad
e=\frac{(N-1)(N-1+\gamma)}{2}.
\]
According to Theorem~\ref{thm:linear-lattice}, we introduce
\[
d_n = an+d = n-(N-1),\quad
e_n = bn+e,
\]
and
\[
\phi^{[n]}(x)
=
\phi(x)+n(n-(N-1)).
\]
The monic orthogonal polynomial sequence associated with $\mathbf u$ then
satisfies recurrence relation
\[
P_{n+1}(x)=(x-a_n)P_n(x)-b_n^{(2)}P_{n-1}(x),
\]
where the coefficient governing regularity is
\[
b_n^{(2)}
=
-\frac{n\,(n-N-1)}{(2n-N-2)(2n-N)}
\left(
\frac{(N-1)^2-\gamma^2}{4}-(n-1)(n-N)
\right)
\]so that
\[
b_{N+1}^{(2)}=0.
\]
Consequently, the recurrence relation terminates at degree $N$, and the functional
$\mathbf u$ is regular, for $\gamma\in\mathbb C$, exactly up to degree $N$. 
\end{eje}

That the polynomials appearing in
Example~\ref{ex:paraKrawtchouk-1classical-form} fall within our framework
is now clear. It remains to characterise all solutions of our problem.

\section{The canonical $1$-classical orthogonal polynomials}\label{canonical}
The following theorem classifies the canonical $1$-classical orthogonal polynomials by the number and multiplicity of the zeros of the polynomial $\phi$ in Definition~\ref{defclass}, mirroring the classical schemes of Bochner: if $\phi$ has no zeros one obtains the $1$-Hermite family;
if $\phi$ has a single zero, the $1$-Laguerre family arises; if $\phi$ has a double zero,
one obtains the $1$-Bessel family; and if $\phi$ has two distinct zeros, one recovers the
$1$-Jacobi family (see Table~\ref{table:repr} below). This new nomenclature avoids the excessive fragmentation of terminology prevalent in the literature. While in Table~\ref{table} we adopted a normalisation aligned with the standard representations of the Hermite, Laguerre, Jacobi, and Bessel polynomials, such a choice is no longer feasible in the present \(1\)-classical setting, as the cases treated in the literature appear under a variety of conventions. Moreover, as will be shown below, families that are commonly treated as canonical in this context, such as the Meixner, Kravchuk, and Charlier polynomials (see, for instance, \cite{GMS95}), can in fact all be realised as particular cases of the \(1\)-Laguerre family. For this reason, Table~\ref{table:repr} is presented in the simplest possible canonical form, with the aim of fixing a convenient reference for future use.

\begin{table}[H]
\centering
\begin{tabular}{>{\columncolor{lightgray}}l|l|l|l}
\rowcolor{lightgray}
Family
& $\Phi(x)$
& $\Psi(x)$
& Conditions on $\alpha, \beta, \gamma \in \mathbb{C}$
\\ \hline
\rule{0pt}{20pt}
\(1\)-Hermite
& $1$
& $\alpha x$
& $\alpha \neq 0,\, -\dps \frac{4}{\alpha}\notin \mathbb{N}^{\times},$
\\
\rule{0pt}{20pt}
\(1\)-Laguerre
& $x$
& $\alpha x+\beta$
& $\alpha \neq 0,\, \beta \neq \dps \frac14\,n(\alpha^{2}-4),$
\\
\rule{0pt}{20pt}
\(1\)-Bessel
& $x^{2}$
& $\alpha x+\beta$
& $-\alpha\notin \mathbb{N}, \, \beta^{2}\neq -\dfrac14\,n(n+\alpha)(2n+\alpha)^{2}$
\\
\rule{0pt}{20pt}
\(1\)-Jacobi
& $x^{2}-\gamma$
& $\alpha x+\beta$
& $-\alpha\notin \mathbb{N},\, \beta^{2}\neq \dps \frac14\,\bigl(4\gamma-n(n+\alpha)\bigr)(2n+\alpha)^{2},\, \gamma\neq 0$
\end{tabular}
\caption{The canonical $1$-classical orthogonal polynomials.}
\label{table:repr}
\end{table}

\begin{theorem}
\label{thm:canonicalreps}Let ${X}:\mathbb{C}\to\mathbb{C}$ be the linear lattice $X(s)=\mathfrak c\,s+\mathfrak d$ with $\mathfrak c\neq0$.
Let $\mathbf u\in\mathcal P'$ be regular, and assume that there exist polynomials
\[
\phi(x)=a x^{2}+b x+c,
\quad
\psi(x)=d x+e,
\]
not both identically zero, such that
\[
\boldsymbol{\mathrm D}_{X}(\phi\,\mathbf u)
=
\boldsymbol{\mathrm S}_{X}(\psi\,\mathbf u).
\]
Set $X_0(s)=s$. Then there exists a regular functional $\mathbf v\in\mathcal P'$ such that
\[
\mathbf v\sim^{*}\mathbf u
\]
and
\[
\boldsymbol{\mathrm D}_{X_0}(\Phi\,\mathbf v)
=
\boldsymbol{\mathrm S}_{X_0}(\Psi\,\mathbf v),
\]
where $(\Phi,\Psi)$ is the canonical pair associated with $(\phi,\psi)$ listed in
Table~\ref{table:repr}. More precisely, one and only one of the following mutually exclusive situations occurs.
\begin{enumerate}[label=\textup{(\roman*)}]
\item\label{item:s-Hermite}
\textup{(\(1\)-Hermite case).}
If $a=b=0$, the affine equivalence can be realised by setting
\[
\mathbf v
=
\left(
\boldsymbol{\tau}_{e\,(d\mathfrak c)^{-1}}
\circ
\boldsymbol{h}_{\mathfrak c^{-1}}
\right)\mathbf u,
\quad
\alpha=\frac{d}{c}\mathfrak c^{2}.
\]

\item\label{item:s-Laguerre}
\textup{(\(1\)-Laguerre case).}
If $a=0$ and $b\neq0$, the affine equivalence can be realised by setting
\[
\mathbf v
=
\left(
\boldsymbol{\tau}_{c\,(b\mathfrak c)^{-1}}
\circ
\boldsymbol{h}_{\mathfrak c^{-1}}
\right)\mathbf u,
\]
with
\[
\alpha=\frac{d}{b}\mathfrak c,
\quad
\beta=\frac{be-dc}{b^{2}}.
\]

\item\label{item:s-Bessel}
\textup{(\(1\)-Bessel case).}
If $a\neq0$ and $b^{2}-4ac=0$, the affine equivalence can be realised by setting
\[
\mathbf v
=
\left(
\boldsymbol{\tau}_{b\,(2a\mathfrak c)^{-1}}
\circ
\boldsymbol{h}_{\mathfrak c^{-1}}
\right)\mathbf u,
\]
with
\[
\alpha=\frac{d}{a},
\quad
\beta=\frac{2ae-bd}{2a^{2}\mathfrak c}.
\]

\item\label{item:s-Jacobi}
\textup{(\(1\)-Jacobi case).}
If $a\neq0$ and $b^{2}-4ac\neq0$, the affine equivalence can be realised by setting
\[
\mathbf v
=
\left(
\boldsymbol{\tau}_{b\,(2a\mathfrak c)^{-1}}
\circ
\boldsymbol{h}_{\mathfrak c^{-1}}
\right)\mathbf u,
\]
with
\[
\alpha=\frac{d}{a},
\quad
\beta=\frac{2ae-bd}{2a^{2}\mathfrak c},
\quad
\gamma=\frac{b^{2}-4ac}{4a^{2}\mathfrak c^{2}}.
\]
\end{enumerate}
\end{theorem}
\begin{proof}
The invariance of the lattice operators under translations of the lattice
variable, proved in Proposition~\ref{proposition:operatorshape}, implies that
$\boldsymbol{\mathrm D}_{X}$ and $\boldsymbol{\mathrm S}_{X}$ depend only on
$\mathfrak c$. We may therefore assume $\mathfrak d=0$ and write
$X(s)=\mathfrak c\,s$. Fix $\xi\in\mathbb C$ and define the functional
\[
\mathbf v
=
(\boldsymbol{\tau}_{\xi}\circ \boldsymbol{h}_{\mathfrak c^{-1}})\mathbf u.
\]
By the covariance of the lattice operators under homotheties and translations
(Proposition~\ref{proposition:operatorshape}), together with the product rules for
transpose maps (Proposition~\ref{proposition:lineartransproposition}\,\textup{(m),(n)}), the
functional $\mathbf v$ satisfies
\begin{equation}\label{eq:master}
\boldsymbol{\mathrm D}_{X_0}\!\bigl(
(\tau_{\xi}h_{\mathfrak c}\phi)\,\mathbf v
\bigr)
=
\boldsymbol{\mathrm S}_{X_0}\!\bigl(
\mathfrak c(\tau_{\xi}h_{\mathfrak c}\psi)\,\mathbf v
\bigr).
\end{equation}
A straightforward computation yields
\[
(\tau_{\xi}h_{\mathfrak c}\phi)(x)
=
a\mathfrak c^{2}x^{2}
+\bigl(b\mathfrak c-2a\mathfrak c^{2}\xi\bigr)x
+\bigl(a\mathfrak c^{2}\xi^{2}-b\mathfrak c\xi+c\bigr),
\]
and
\[
\mathfrak c(\tau_{\xi}h_{\mathfrak c}\psi)(x)
=
d\mathfrak c^{2}x+e\mathfrak c-d\mathfrak c^{2}\xi.
\]
We now analyse separately the possible values of the coefficients
$a,b,c$, which lead to the canonical forms listed in
Table~\ref{table:repr}.

\noindent\textup{(i) $a=b=0$.}
Then $\phi=c$ with $c\neq0$, and regularity implies $d\neq0$.
This corresponds to the degenerate situation in which both the quadratic and
linear parts of $\phi$ vanish, so that the functional equation involves only a
constant coefficient on the left-hand side. Equation~\eqref{eq:master} becomes
\[
\boldsymbol{\mathrm D}_{X_0}(c\,\mathbf v)
=
\boldsymbol{\mathrm S}_{X_0}\!\bigl(
(d\mathfrak c^{2}x+e\mathfrak c-d\mathfrak c^{2}\xi)\mathbf v
\bigr).
\]
Dividing by $c$ and choosing
\[
\xi=\frac{e}{d\mathfrak c},
\]
the constant term is eliminated. This normalisation isolates the essential
linear dependence on the lattice variable, yielding
\[
\boldsymbol{\mathrm D}_{X_0}(\mathbf v)
=
\boldsymbol{\mathrm S}_{X_0}\!\bigl(
(\alpha x)\mathbf v
\bigr),
\quad
\alpha=\frac{d}{c}\mathfrak c^{2},
\]
which coincides with the \(1\)-Hermite canonical form.

\smallskip
\noindent\textup{(ii) $a=0$ and $b\neq0$.}
In this case $\phi$ is affine but nonconstant. A suitable translation allows us
to remove the constant term, while preserving the linear structure that
characterises this situation. Indeed,
\[
(\tau_{\xi}h_{\mathfrak c}\phi)(x)=b\mathfrak c x+c-b\mathfrak c\xi.
\]
Equation~\eqref{eq:master} reads
\[
\boldsymbol{\mathrm D}_{X_0}\!\bigl(
(b\mathfrak c x+c-b\mathfrak c\xi)\mathbf v
\bigr)
=
\boldsymbol{\mathrm S}_{X_0}\!\bigl(
(d\mathfrak c^{2}x+e\mathfrak c-d\mathfrak c^{2}\xi)\mathbf v
\bigr).
\]
Dividing by $b\mathfrak c$ and choosing
\[
\xi=\frac{c}{b\mathfrak c},
\]
the constant term vanishes, and we obtain
\[
\boldsymbol{\mathrm D}_{X_0}(x\,\mathbf v)
=
\boldsymbol{\mathrm S}_{X_0}\!\bigl(
(\alpha x+\beta)\mathbf v
\bigr),
\]
with
\[
\alpha=\frac{d}{b}\mathfrak c,
\quad
\beta=\frac{be-dc}{b^{2}},
\]
which is the \(1\)-Laguerre canonical form.

\smallskip
\noindent\textup{(iii) $a\neq0$ and $b^{2}-4ac=0$.}
In this case $\phi$ is genuinely quadratic and has a double root. Translating
this root to the origin simplifies the expression of $\phi$ and reveals the
underlying quadratic structure. Choose
\[
\xi=\frac{b}{2a\mathfrak c}.
\]
Then the linear and constant terms in $(\tau_{\xi}h_{\mathfrak c}\phi)$ vanish, and
\[
(\tau_{\xi}h_{\mathfrak c}\phi)(x)=a\mathfrak c^{2}x^{2}.
\]
Dividing~\eqref{eq:master} by $a\mathfrak c^{2}$ yields
\[
\boldsymbol{\mathrm D}_{X_0}(x^{2}\mathbf v)
=
\boldsymbol{\mathrm S}_{X_0}\!\bigl(
(\alpha x+\beta)\mathbf v
\bigr),
\]
with
\[
\alpha=\frac{d}{a},
\quad
\beta=\frac{2ae-bd}{2a^{2}\mathfrak c},
\]
which is the \(1\)-Bessel canonical form.

\smallskip
\noindent\textup{(iv) $a\neq0$ and $b^{2}-4ac\neq0$.}
With the same choice $\xi=\frac{b}{2a\mathfrak c}$,
\[
(\tau_{\xi}h_{\mathfrak c}\phi)(x)
=
a\mathfrak c^{2}\bigl(x^{2}-\gamma\bigr),
\quad
\gamma=\frac{b^{2}-4ac}{4a^{2}\mathfrak c^{2}}.
\]
Dividing~\eqref{eq:master} by $a\mathfrak c^{2}$ gives
\[
\boldsymbol{\mathrm D}_{X_0}\!\bigl((x^{2}-\gamma)\mathbf v\bigr)
=
\boldsymbol{\mathrm S}_{X_0}\!\bigl(
(\alpha x+\beta)\mathbf v
\bigr),
\]
with
\[
\alpha=\frac{d}{a},
\quad
\beta=\frac{2ae-bd}{2a^{2}\mathfrak c},
\]
which is the \(1\)-Jacobi canonical form. 

Finally, regularity is invariant under affine transformations of the form
$\boldsymbol{\tau}_{\xi}\circ\boldsymbol{h}_{\mathfrak c^{-1}}$, as a consequence
of the covariance of the operators $\boldsymbol{\mathrm D}_X$ and
$\boldsymbol{\mathrm S}_X$ under translations and homotheties
(Proposition~\ref{proposition:operatorshape}), together with the product rules for
transpose maps (Proposition~\ref{proposition:lineartransproposition}). Therefore, $\mathbf v$ is regular if and
only if the associated canonical functional is regular. To make this explicit,
we apply the regularity criterion of
Theorem~\ref{thm:linear-lattice} to the canonical pairs $(\Phi,\Psi)$. On the $X_0$-lattice, regularity is characterised by
\[
d_n\neq 0
\quad\text{and}\quad
\Phi^{[n]}\!\left(-\frac{e_n}{d_{2n}}\right)\neq 0,
\]
where $d_n=d+na$ and $e_n=e+nb$. We illustrate the computation in the most involved
case, namely the \(1\)-Jacobi canonical pair
\[
(\Phi,\Psi)=(x^{2}-\gamma,\alpha x+\beta), \quad \gamma\neq 0.
\]
In this case $(a,b,c,d,e)=(1,0,-\gamma,\alpha,\beta)$, hence
\[
d_n=n+\alpha,\quad e_n=\beta,\quad
\Phi^{[n]}(x)=x^{2}-\gamma+\frac{n}{4}(n+\alpha).
\]
Therefore $d_n\neq 0$ for all $n\in\mathbb N$ is equivalent to
$-\alpha\notin\mathbb N$. Moreover $d_{2n}=2n+\alpha$, so
\[
\Phi^{[n]}\!\left(-\frac{e_n}{d_{2n}}\right)
=
\left(\frac{\beta}{2n+\alpha}\right)^{2}
-\gamma+\frac{n}{4}(n+\alpha)\neq 0,
\quad n\in\mathbb N,
\]
equivalently,
\[
4\beta^{2}-4\gamma(2n+\alpha)^{2}
+n(n+\alpha)(2n+\alpha)^{2}\neq 0,
\quad n\in\mathbb N.
\]
This yields exactly the regularity conditions stated in
Table~\ref{table:repr} for the \(1\)-Jacobi case. The remaining canonical cases
follow from the same criterion and lead to simpler expressions. The resulting
conditions are collected in Table~\ref{table:repr}. The proof is complete.
\end{proof}

For future reference, the explicit formulas for the
recurrence coefficients $a_n$ and $b_n^{(\mathfrak c)}$ of
Theorem~\ref{thm:linear-lattice} are recorded below, under the regularity conditions
of Table~\ref{table:repr}.

\begin{table}[H]
\centering
\renewcommand{\arraystretch}{1.25}
\begin{tabular}{>{\columncolor{lightgray}}l|l}
\rowcolor{lightgray}
Family & $a_n$ \\ \hline
\rule{0pt}{22pt}
$1$--Hermite
& $0$
\\
\rule{0pt}{24pt}
$1$--Laguerre
& $-\dps\frac{2n+\beta}{\alpha}$
\\
\rule{0pt}{30pt}
$1$--Bessel
& $\beta\!\left(\dps\frac{n}{2n+\alpha-2}-\dps\frac{n+1}{2n+\alpha}\right)$
\\
\rule{0pt}{32pt}
$1$--Jacobi
& $\beta\!\left(\dps\frac{n}{2n+\alpha-2}-\dps\frac{n+1}{2n+\alpha}\right)$
\\
\end{tabular}
\caption{The explicit form of the coefficients $a_n$ in the recurrence relation
of Theorem~\ref{thm:linear-lattice} under the regularity conditions
of Table~\ref{table:repr}.}
\label{table:rec-canonical-a}
\end{table}

\begin{table}[H]
\centering
\renewcommand{\arraystretch}{1.25}
\begin{tabular}{>{\columncolor{lightgray}}l|l}
\rowcolor{lightgray}
Family & $b_n^{(\mathfrak c)}$ \\ \hline
\rule{0pt}{22pt}
$1$--Hermite
& $-n\!\left(\dps\frac1{\alpha}+\dps\frac{n-1}{4}\right)$
\\
\rule{0pt}{24pt}
$1$--Laguerre
& $-\dps\frac{n}{4\alpha^{2}}\Bigl((n-1)(\alpha^{2}-4)+4\beta\Bigr)$
\\
\rule{0pt}{30pt}
$1$--Bessel
& $-\dps\frac{n(n-2+\alpha)}{(2n-3+\alpha)(2n-1+\alpha)}
\left(\dps\frac{n-1}{4}(n-1+\alpha)+\dps\frac{\beta^{2}}{(2n-2+\alpha)^{2}}\right)$
\\
\rule{0pt}{32pt}
$1$--Jacobi
& $-\dps\frac{n(n-2+\alpha)}{(2n-3+\alpha)(2n-1+\alpha)}
\left(\dps\frac{n-1}{4}(n-1+\alpha)+\dps\frac{\beta^{2}}{(2n-2+\alpha)^{2}}-\gamma\right)$
\\
\end{tabular}
\caption{The explicit form of the coefficients $b_n^{(\mathfrak c)}$ in the recurrence relation
of Theorem~\ref{thm:linear-lattice} under the regularity conditions
of Table~\ref{table:repr}.}
\label{table:rec-canonical-b}
\end{table}

\begin{remark}
The reader may wonder why an additional parameter appears in each family when
compared with the Bochner setting. This feature is in fact intrinsic
to the lattice formulation. Indeed, when working on a linear lattice
\(X(s)=\mathfrak c\,s+\mathfrak d\), Proposition~\ref{proposition:operatorshape} shows
that the associated operators are invariant under translations of the lattice
variable, so that the shift parameter \(\mathfrak d\) plays no role. In contrast,
the scaling parameter \(\mathfrak c\) enters in a nontrivial way. Since the canonical forms are written on the normalised lattice
\(X_0(s)=s\), the extra parameter appearing in each family precisely reflects
this scaling effect. In particular, the notion of \(1\)-classicality is not an intrinsic property of the functional alone, but rather depends on the choice of the underlying linear lattice. Owing to this dependence, the \(1\)-Bessel and \(1\)-Jacobi cases could in fact be merged; nevertheless, we prefer to preserve the harmony with Bochner’s original framework.
\end{remark}

\begin{definition}
\label{def:canonicalfamilies}
The canonical families corresponding to cases~\textup{(i)}--\textup{(iv)} of
Theorem~\ref{thm:canonicalreps} are denoted, respectively, by
\[
\bigl(H_n^{(\alpha)}\bigr)_{n\ge 0},\quad
\bigl(L_n^{(\alpha,\beta)}\bigr)_{n\ge 0},\quad
\bigl(B_n^{(\alpha,\beta)}\bigr)_{n\ge 0},\quad
\bigl(J_n^{(\alpha,\beta,\gamma)}\bigr)_{n\ge 0}.
\]
Their associated canonical regular functionals are denoted by
\[
\mathbf u_{\mathcal H_{\mathfrak c}}^{(\alpha)},\quad
\mathbf u_{\mathcal L_{\mathfrak c}}^{(\alpha,\beta)},\quad
\mathbf u_{\mathcal B_{\mathfrak c}}^{(\alpha,\beta)},\quad
\mathbf u_{\mathcal J_{\mathfrak c}}^{(\alpha,\beta,\gamma)},
\]
in the same order.
\end{definition}

\begin{eje}[Para-Krawtchouk polynomials revisited]\label{K2}
As shown in Example~\ref{ex:paraKrawtchouk-1classical-form}, the associated
functional $\mathbf u$ is $1$-classical with respect to the lattice $X(s)=2s$ and
satisfies the functional equation
\[
\boldsymbol{\mathrm D}_{X}(\phi\,\mathbf u)
=
\boldsymbol{\mathrm S}_{X}(\psi\,\mathbf u).
\]
In this setting, the polynomials $\phi$ and $\psi$ are explicitly determined in
Example~\ref{ex:paraKrawtchouk-1classical-form}. Since
\[
b^{2}-4ac=(N-1+\gamma)\bigl(\gamma-(N-1)\bigr),
\]
the polynomial $\phi$ has two distinct roots whenever $\gamma\neq\pm(N-1)$.
In this generic situation, the functional $\mathbf u$ falls into the $1$-Jacobi
case of Theorem~\ref{thm:canonicalreps}. To obtain a canonical representative, one
proceeds exactly as in case~\textup{(iv)} of that theorem: one first eliminates the
linear term of $\phi$ by a suitable translation of the lattice variable, and then
normalises the lattice slope by a homothety. More precisely, setting
$\xi=b/(2a\mathfrak c)$ and applying the affine transformation
$\boldsymbol{\tau}_{\xi}\circ \boldsymbol{h}_{\mathfrak c^{-1}}$, one obtains an
affine equivalent functional
\begin{equation*}
\mathbf v
=
\left(\boldsymbol{\tau}_{\frac{b}{2a\mathfrak c}}\circ
\boldsymbol{h}_{\mathfrak c^{-1}}\right)\mathbf u
=
\left(\boldsymbol{\tau}_{-\frac{N-1+\gamma}{4}}\circ
\boldsymbol{h}_{\frac12}\right)\mathbf u=\mathbf{u_{\mathcal J_{\mathfrak c}}}^{(\alpha_0,\beta_0,\gamma_0)},
\end{equation*}
for which
\begin{equation*}
\boldsymbol{\mathrm D}_{X_0}\!\bigl((x^{2}-\gamma_0)\,\mathbf v\bigr)
=
\boldsymbol{\mathrm S}_{X_0}\!\bigl((\alpha_0 x+\beta_0)\,\mathbf v\bigr),
\end{equation*}
with
\begin{align*}
\alpha_0&=\frac{d}{a}=-(N-1),
\quad
\beta_0=\frac{2ae-bd}{2a^{2}\mathfrak c}=0,\\[7pt]
\gamma_0&=\frac{b^{2}-4ac}{4a^{2}\mathfrak c^{2}}
=\frac{(N-1+\gamma)\,(\gamma-(N-1))}{16}.
\end{align*}
Using the above reduction, the orthogonal polynomials associated with $\mathbf u$
can be written explicitly in terms of the canonical $1$-Jacobi family. More precisely,
one has
\[
p_n(x)
=
2^{\,n}\,
J_n^{(\alpha_0,0,\gamma_0)}\!\left(
\frac{x}{2}+\frac{N-1+\gamma}{4}
\right),
\quad n=0,1,\dots,N.
\]

When $\gamma=\pm(N-1)$, the discriminant vanishes and $\phi$ has a double root. In this
degenerate situation one is in the $1$-Bessel case of
Theorem~\ref{thm:canonicalreps}. Indeed, $b^{2}-4ac=0$, so that the reduction to
canonical form simplifies. As in the $1$-Jacobi case, we apply the affine transformation
\[
\mathbf v
=
\left(\boldsymbol{\tau}_{\xi}\circ\boldsymbol{h}_{\mathfrak c^{-1}}\right)\mathbf u=\mathbf{u_{\mathcal B_{\mathfrak c}}}^{(\alpha,\beta)},
\quad
\xi=\frac{b}{2a\mathfrak c}
=
-\frac{N-1+\gamma}{4}.
\]
A direct computation shows that
\[
(\tau_{\xi}h_{\mathfrak c}\phi)(x)=a\mathfrak c^{2}x^{2},
\]
that is, both the linear and constant terms of $\phi$ vanish after translation,
leaving a pure quadratic polynomial. Dividing the resulting functional equation by
$a\mathfrak c^{2}$, one obtains the canonical $1$-Bessel equation
\[
\boldsymbol{\mathrm D}_{X_0}\!\bigl(x^{2}\mathbf v\bigr)
=
\boldsymbol{\mathrm S}_{X_0}\!\bigl((\alpha_0 x+\beta_0)\mathbf v\bigr),
\]
with
\[
\alpha_0=\frac{d}{a}=-(N-1),
\quad
\beta_0=\frac{2ae-bd}{2a^{2}\mathfrak c}=0.
\]
In this case, corresponding necessarily to $N=2$, one has
\[
p_n(x)
=
2^{\,n}\,
B_n^{(\alpha_0,0)}\!\left(
\frac{x}{2}+\frac{N-1+\gamma}{4}
\right),
\quad n=0,1,2.
\]
\end{eje}

\begin{eje}[Garc\'ia-Marcell\'an-Salto Classification]
\label{ex:GMS95-Table1-to-canonical-integrated}
We revisit the families introduced in \cite[Table~1]{GMS95}, which are defined on
the lattice $X_0(s)=s$ by prescribing a pair of polynomials
\[
\phi(x)=ax^{2}+bx+c,
\quad
\psi(x)=dx+e,
\]
such that the associated functional $\mathbf u\in\mathcal P'$ satisfies the functional equation
\[
\boldsymbol{\Delta}(\phi\,\mathbf u)=\psi\,\mathbf u.
\]
By Proposition~\ref{proposition:disc-to-centered}, this relation is equivalent to 
\[
\boldsymbol{\mathrm D}_{X_0}\!\bigl(\Phi\,\mathbf u\bigr)
=
\boldsymbol{\mathrm S}_{X_0}\!\bigl(\Psi\,\mathbf u\bigr),
\quad
\Phi(x)=2\phi(x)+\psi(x),\quad \Psi(x)=2\psi(x),
\]
and therefore all four families fall within the scope of
Theorem~\ref{thm:canonicalreps}. The corresponding canonical families, parameters,
and regularity properties follow directly from this theorem and
Table~\ref{table:repr}. The detailed verifications are omitted, since, quite
independently of the functional viewpoint advocated in \cite{GMS95}, the resulting
families reduce to particular cases of the positive definite setting already
studied in the literature and revisited in a later example.

For the Charlier case of \cite[Table~1]{GMS95}, one has
\(\phi(x)=1\) and \(\psi(x)=-x+e\),
and therefore
\[
\Phi(x)=2\phi(x)+\psi(x)=-x+(e+2),
\]
which has degree~$1$. In this case the regularity conditions are automatically
satisfied for all parameter values, and therefore the associated orthogonal
polynomial family is infinite. Under these circumstances, the functional belongs
to the canonical $1$-Laguerre family with parameters $\alpha=2$ and $\beta=4$,
and the corresponding monic orthogonal polynomials are given by
\[
p_n(x)=L_n^{(2,4)}(x-e-2), \quad n\in\mathbb N.
\]

For the Meixner case of \cite[Table~1]{GMS95}, one has
\(\phi(x)=x\) and \(\psi(x)=-x+e\), and
so that
\[
\Phi(x)=2\phi(x)+\psi(x)=x+e,
\]
again of degree~$1$. Regularity holds whenever $e\neq0$, and in this situation the
associated orthogonal polynomial family is infinite. The functional then belongs
to the canonical $1$-Laguerre family with parameters $\alpha=-2$ and $\beta=5e$,
and the corresponding monic orthogonal polynomials are given by
\[
p_n(x)=L_n^{(-2,\,5e)}(x+e), \quad n\in\mathbb N.
\]

For the Krawtchouk case of \cite[Table~1]{GMS95}, one has
\(\phi(x)=x\) and \(\psi(x)=2x+e\),
and hence
\[
\Phi(x)=2\phi(x)+\psi(x)=4x+e,
\]
which has degree~$1$, but regularity is not automatic. The regularity criterion
shows that the recurrence coefficient vanishes at a finite index whenever
$e=-3N$ for some $N\in\mathbb N$. In this situation, the recurrence relation
breaks down at index $N+1$, and the associated orthogonal polynomial family exists
only up to degree $N$. When regularity holds, the functional belongs to the
canonical $1$-Laguerre family with parameters $\alpha=1$ and $\beta=e/4$, and the
corresponding monic orthogonal polynomials are given by
\[
p_n(x)=L_n^{(1,\,e/4)}\!\left(x+\frac{e}{4}\right),
\quad n=0,1,\dots,N,
\]
where $N$ is the largest index for which the recurrence coefficient does not
vanish. 

Finally, for the Hahn case of \cite[Table~1]{GMS95}, one has
\(\phi(x)=x^{2}+x+1\) and \(\psi(x)=-2N\,x+e\),
so that
\[
\Phi(x)=2\phi(x)+\psi(x)=2x^{2}+2(1-N)x+(2+e),
\]
which has degree~$2$, and the situation is structurally different. Independently
of the value of the discriminant of $\Phi$, regularity is finite: the regularity
coefficient vanishes at a finite index, so that the recurrence
relation terminates and the associated orthogonal polynomial family exists only
up to a finite degree. If the discriminant of $\Phi$ is nonzero, the functional
belongs to the canonical $1$-Jacobi family with parameters
\[
\alpha_0=-2N,\quad
\beta_0=e-N(N-1),\quad
\gamma_0=\frac{(1-N)^{2}-2(2+e)}{4},
\]
and the associated monic orthogonal polynomials are given by
\[
p_n(x)=J_n^{(\alpha_0,\beta_0,\gamma_0)}\!\left(x+\frac{1-N}{2}\right),
\quad n=0,1,\dots,2N-1.
\]
If the discriminant vanishes, the functional instead belongs to the canonical
$1$-Bessel family with parameters
\[
\alpha_0=-2N,\quad \beta_0=e-N(N-1),
\]
and the associated monic orthogonal polynomials are given by
\[
p_n(x)=B_n^{(\alpha_0,\beta_0)}\!\left(x+\frac{1-N}{2}\right),
\quad n=0,1,\dots,2N-1.
\]

\end{eje}

Example~\ref{ex:GMS95-Table1-to-canonical-integrated} shows that the distinction
between the Charlier, Meixner and Krawtchouk cases in \cite[Table~1]{GMS95} is
largely superficial: all three fall within the same canonical $1$-Laguerre
family. The absence of the $1$-Hermite case in
Example~\ref{ex:GMS95-Table1-to-canonical-integrated} will be better
understood in the following section.

\section{The positive definite case}\label{pd}
We work throughout under the hypotheses of Theorem~\ref{thm:linear-lattice} and
Theorem~\ref{thm:canonicalreps}.
In particular, for each canonical functional listed in
Table~\ref{table:repr}, the associated monic orthogonal polynomial sequence is defined by the recurrence relation of
Theorem~\ref{thm:linear-lattice}, with recurrence coefficients
$a_n$ and $b_n^{(\mathfrak c)}$ given explicitly in~Tables~\ref{table:rec-canonical-a} and \ref{table:rec-canonical-b}. A regular functional $\mathbf u\in\mathcal P'$ is said to be {positive definite}
if
\[
b_n^{(\mathfrak c)}>0,
\quad n\in\mathbb N^{\times}.
\]
If this condition holds only up to some finite index, the same recurrence relation
defines a finite monic orthogonal polynomial sequence in the sense of
Theorem~\ref{thm:linear-lattice}. It is well known that positive definiteness implies the existence of an integral
representation of $\mathbf u$ with respect to a positive Borel measures on $\mathbb{R}$; see, for instance, \cite[Theorem~2.9]{CP25}. Although the present work is not concerned with such positive representations, the mere
knowledge of their existence guarantees that the associated orthogonal polynomials
enjoy the usual analytic properties, such as real and simple zeros, interlacing of
consecutive zeros, and extremal characterisations, without any need to make the
underlying measure explicit. Finally, note that positivity forces the recurrence coefficients to be real.
Indeed, if $b_n^{(\mathfrak c)}>0$ for all $n\in\mathbb{N}^{\times}$, then
$b_n^{(\mathfrak c)}\in\mathbb R$.
Since the coefficients $a_n$ and $b_n^{(\mathfrak c)}$ are algebraic expressions in
the parameters of the canonical pair $(\Phi,\Psi)$, the reality of
$b_n^{(\mathfrak c)}$ for all $n$ implies the reality of $a_n$ as well.
Accordingly, whenever positivity is discussed, we tacitly restrict to real
parameters.

For the functional $\mathbf u_{\mathcal H_{\mathfrak c}}^{(\alpha)}$, the recurrence
coefficient $b_{n+1}^{(\mathfrak c)}$ is given in
Table~\ref{table:rec-canonical-b}. If $\alpha\in\mathbb R\setminus\{0\}$, then
\[
\frac{1}{\alpha}+\frac{n}{4}
=\frac{n}{4}+O(1),\quad n\to\infty,
\]
and consequently
\[
b_{n+1}^{(\mathfrak c)}
=-\frac{n^{2}}{4}+O(n)<0
\]
for all sufficiently large $n$. It follows that the functional cannot be positive
definite. Moreover, if $\alpha<0$, the inequality
$b_{n+1}^{(\mathfrak c)}>0$ holds only for finitely many indices $n$, which implies
that the corresponding orthogonality is finite. All this explains why it does not appear as a possibility in
Example~\ref{ex:GMS95-Table1-to-canonical-integrated} and shows that the
functional approach envisaged in~\cite{GMS95} did not lead to any
substantive structural extension.

For the functional $\mathbf u_{\mathcal L_{\mathfrak c}}^{(\alpha,\beta)}$ with
$\alpha\neq0$, the recurrence coefficient $b_{n+1}^{(\mathfrak c)}$ is given in
Table~\ref{table:rec-canonical-b}. Assume $\alpha,\beta\in\mathbb R$. Since
\[
-\frac{n+1}{4\alpha^{2}}<0,
\]
the sign of $b_{n+1}^{(\mathfrak c)}$ is determined by the affine function
\[
n\longmapsto n(\alpha^{2}-4)+4\beta .
\]
If $\alpha^{2}\le4$, this function is nonincreasing, and
$b_{n+1}^{(\mathfrak c)}>0$ for all $n\in\mathbb N$ if and only if $\beta<0$.
In this case the functional is positive definite.
If $\alpha^{2}>4$, the above expression tends to $+\infty$ as $n\to\infty$, and
$b_{n+1}^{(\mathfrak c)}$ eventually becomes negative.
When $\beta<0$, positivity holds only up to a finite index, which yields a finite
orthogonal polynomial sequence.

For the functional $\mathbf u_{\mathcal B_{\mathfrak c}}^{(\alpha,\beta)}$, the
recurrence coefficient $b_{n+1}^{(\mathfrak c)}$ is given in
Table~\ref{table:rec-canonical-b}. Assume $\alpha,\beta\in\mathbb R$ and that the
regularity conditions in Table~\ref{table:repr} are satisfied. The rational
prefactor is positive for all sufficiently large $n$ and converges to $1/4$ as
$n\to\infty$. Moreover,
\[
\frac{n}{4}(n+\alpha)
+\frac{\beta^{2}}{(2n+\alpha)^{2}}
=
\frac{n^{2}}{4}+O(n),
\quad n\to\infty,
\]
so that the bracketed term is positive for all sufficiently large $n$.
Consequently,
\[
b_{n+1}^{(\mathfrak c)}<0
\]
for all sufficiently large $n$, and the functional cannot be positive definite.
At most, positivity may hold up to a finite index.

For the functional $\mathbf u_{\mathcal J_{\mathfrak c}}^{(\alpha,\beta,\gamma)}$
with $\gamma\neq0$, the recurrence coefficient $b_{n+1}^{(\mathfrak c)}$ is given in
Table~\ref{table:rec-canonical-b}. Assume $\alpha,\beta,\gamma\in\mathbb R$ and that
regularity holds. As in the $1$--Bessel case, the rational prefactor is eventually
positive and converges to $1/4$. Moreover,
\[
\frac{n}{4}(n+\alpha)
+\frac{\beta^{2}}{(2n+\alpha)^{2}}
-\gamma
=
\frac{n^{2}}{4}+O(n),
\quad n\to\infty,
\]
which is positive for all sufficiently large $n$. Consequently,
\[
b_{n+1}^{(\mathfrak c)}<0
\]
for all sufficiently large $n$, and the functional is not positive definite.
At most, positivity may hold up to a finite index. Setting $\gamma=0$ recovers the
$1$-Bessel family.

\begin{table}[H]
\centering
\begin{tabular}{>{\columncolor{lightgray}}l|l|l}
\rowcolor{lightgray}
Family
& Infinite orthogonality
& Finite orthogonality
\\ \hline
\rule{0pt}{18pt}
\(1\)-Hermite
& none
& occurs for $\alpha<0$
\\
\rule{0pt}{18pt}
\(1\)-Laguerre
& $\alpha^2\le4,\ \beta<0$
& occurs if $\alpha^2>4$ and $\beta<0$
\\
\rule{0pt}{18pt}
\(1\)-Bessel
& none
& occurs (for admissible parameters)
\\
\rule{0pt}{18pt}
\(1\)-Jacobi
& none
& occurs (for admissible parameters)
\end{tabular}
\caption{Infinite and finite orthogonality for the canonical $1$--classical functionals
in the positive definite case.}
\label{table:positivity}

\end{table}

Let us return to Example~\ref{ex:paraKrawtchouk-1classical-form} in the positive definite setting, which corresponds to the particular case treated in \cite{VZ12}.
\begin{eje}[Para-Krawtchouk polynomials revisited]
\label{rem:paraK-to-Hahn-Theorem52}
In the framework of Examples~\ref{ex:paraKrawtchouk-1classical-form} and~\ref{K2},
the recurrence coefficients $b_n^{(2)}$ are given explicitly in terms of the
parameters $N$ and $\gamma$ by Theorem~\ref{thm:linear-lattice}, with the details
worked out in Example~\ref{ex:paraKrawtchouk-1classical-form}. In particular, the functional is positive definite
if $\gamma\in\mathbb R$ and $0<\gamma<2$. It is worth emphasizing that
\cite{KLS10} provides a complete analysis of the positive definite case. To compare \eqref{paraK} with the classification displayed in
\cite[Theorem~5.2]{KLS10}, we decimate the step-$2$ equation to the even sublattice.
Set
\[
y_n(s)=p_n(2s),\quad s\in\mathbb Z.
\]
Then \eqref{paraK} becomes a step-$1$ equation and, using our forward difference
$\Delta$ and $\Delta^{2}$, it can be written equivalently $($after the index shift
$s\mapsto s-1$$)$ in the form of \cite[Theorem~5.2.1]{KLS10} as
\begin{equation}\label{eq:paraK-NU-form}
\bigl(a s^{2}+2 b s+c\bigr)\,(\Delta^{2} y_n)(s)
+\bigl(2 d s+e\bigr)\,(\Delta y_n)(s)
=
n\bigl(a(n-1)+2 d\bigr)\,y_n(s+1),
\end{equation}
with the explicit identification
\[
a=2,\quad
2b=6-2N-\gamma,\quad
c=\frac{N^{2}+N\gamma-6N-3\gamma+9}{2},
\]
and
\[
2d=2(1-N),\quad
e=\frac{N^{2}+N\gamma-6N-\gamma+5}{2}.
\]
Since $a\neq0$, equation~\eqref{eq:paraK-NU-form} belongs to the Hahn block
$($Cases~III$)$ of \cite[Theorem~5.2]{KLS10}. Moreover, the endpoint vanishing encoded in
Example~\ref{ex:paraKrawtchouk-1classical-form} forces finite support, hence one of
the finite Hahn cases. After an inessential affine change of the discrete variable
(allowed in our framework by Proposition~\ref{proposition:affine-invariance}),
equation~\eqref{eq:paraK-NU-form} is therefore identified with the finite Hahn case
in \cite[Theorem~5.2]{KLS10} $($Case~IIIa1$)$. For $N=2$ and $\gamma=1$ in Example~\ref{ex:paraKrawtchouk-1classical-form}, the family
reduces to $\{P_0,P_1,P_2\}$ and, while remaining within the Hahn block in the sense of
\cite{KLS10}, it is not worth pursuing this isolated case any further. \end{eje}

\section{From \(1\)-classicality to \(0\)-classicality}\label{01}
In this section, we interpret the classical, or Bochner, framework, henceforth called \(0\)-classical, as the \(\mathfrak c \to 0\) limit of the \(1\)-classical setting, where \(\mathfrak c\) denotes the linear parameter associated with the lattice. The limit is taken in the weak topology $\sigma(\mathcal P',\mathcal P)$ along suitable admissible sets $U\subset\mathbb C^\times$ satisfying $0\in\overline U$, where $\overline U$ denotes the closure of $U$. The existence of such sets is guaranteed and is discussed below.

\begin{lemma}\label{lem:weak-limit-exists}
Let $U\subset\mathbb{C}^{\times}$ be a set with $0\in\overline U$, and fix
$\mathfrak d\in\mathbb{C}$. For each $\mathfrak c\in U$ let
$X(s)=\mathfrak c s+\mathfrak d$. Let
$\{\mathbf u_{\mathfrak c}:\mathfrak c\in U\}\subset\mathcal P'$ be a family of
nonzero linear functionals such that, for every $\mathfrak c\in U$,
\[
\boldsymbol{\mathrm D}_{X}(\phi\,\mathbf u_{\mathfrak c})
=
\boldsymbol{\mathrm S}_{X}(\psi\,\mathbf u_{\mathfrak c}),
\]
where $\phi(x)=ax^2+bx+c$ and $\psi(x)=dx+e$ are fixed polynomials. Set
$d_n=an+d$ and assume $d_n\neq 0$ for all $n\in\mathbb{N}$. After scaling each
$\mathbf u_{\mathfrak c}$ by a nonzero complex constant, assume that
\[
\langle \mathbf u_{\mathfrak c},1\rangle=1,\quad \mathfrak c\in U.
\]
Then there exists a unique functional $\mathbf u_0\in\mathcal P'$ such that
\[
\mathbf u_{\mathfrak c}\longrightarrow \mathbf u_0
\]
in $\sigma(\mathcal P',\mathcal P)$ as $\mathfrak c\to0$ within $U$.
Moreover, for every $n\in\mathbb{N}$ the limit
\[
\mu_n=\lim_{\substack{\mathfrak c\to0\\ \mathfrak c\in U}}
\langle \mathbf u_{\mathfrak c},x^n\rangle
\]
exists, is finite, and satisfies $\langle \mathbf u_0,x^n\rangle=\mu_n$.
\end{lemma}

\begin{proof}
For each $n\in\mathbb{N}$, we pair the functional equation against $x^n$.
By transposition of the operators $\boldsymbol{\mathrm D}_{X}$ and
$\boldsymbol{\mathrm S}_{X}$, this yields
\begin{equation}\label{eq:paired-U}
\bigl\langle \mathbf u_{\mathfrak c},\,
\phi\,\mathrm D_{X}(x^n)
+\psi\,\mathrm S_{X}(x^n)\bigr\rangle=0 .
\end{equation}
Writing $h=\mathfrak c/2$ and expanding binomially (see Remark~\ref{rem:DX-SX-limit}), one obtains
\[
\mathrm S_{X}(x^n)
=
\sum_{k=0}^{\lfloor n/2\rfloor}
\binom{n}{2k}h^{2k}x^{n-2k},
\quad
\mathrm D_{X}(x^n)
=
\sum_{k=0}^{\lfloor (n-1)/2\rfloor}
\binom{n}{2k+1}h^{2k}x^{n-(2k+1)} .
\]
Define the polynomial $Q_{n,\mathfrak c}$ of degree at most $n+1$ by
\[
Q_{n,\mathfrak c}(x)
=
\phi(x)\mathrm D_{X}(x^n)
+\psi(x)\mathrm S_{X}(x^n).
\]
A direct coefficient comparison gives
\[
Q_{n,\mathfrak c}(x)
=
d_n\,x^{n+1}+(bn+e)\,x^n
+\sum_{j=0}^{n-1}A_{n,j}(\mathfrak c^2)\,x^j ,
\]
where $d_n\neq0$ and
the coefficients $A_{n,j}(\mathfrak c^2)$ are polynomials in $\mathfrak c^2$, with the convention $\binom{n}{k}=0$ for $k>n$, given by
\begin{equation}\label{eq:Anj-explicit}
\begin{aligned}
A_{n,j}(\mathfrak c^{2})
&=
\mathbf{1}_{\,n+1-j\text{ even}}
\left(
a\binom{n}{n+2-j}+d\binom{n}{n+1-j}
\right)\left(\frac{\mathfrak c^{2}}{4}\right)^{\frac{n+1-j}{2}}
\\[6pt]
&\quad+
\mathbf{1}_{\,n-j\text{ even}}
\left(
b\binom{n}{n+1-j}+e\binom{n}{n-j}
\right)\left(\frac{\mathfrak c^{2}}{4}\right)^{\frac{n-j}{2}}
\\[6pt]
&\quad+
\mathbf{1}_{\,n-1-j\text{ even}}
\left(
c\binom{n}{n-j}
\right)\left(\frac{\mathfrak c^{2}}{4}\right)^{\frac{n-1-j}{2}},
\end{aligned}
\end{equation}
where \(\mathbf{1}_{\,C}\) denotes the indicator function of the condition \(C\).
Substituting this expansion into \eqref{eq:paired-U} and writing
$\mu_k(\mathfrak c)=\langle\mathbf u_{\mathfrak c},x^k\rangle$, we obtain the triangular
recursion
\[
d_n\,\mu_{n+1}(\mathfrak c)
+(bn+e)\,\mu_n(\mathfrak c)
+\sum_{j=0}^{n-1}A_{n,j}(\mathfrak c^2)\,\mu_j(\mathfrak c)=0 .
\]
Since $d_n\neq0$ for all $n$ and $\mu_0(\mathfrak c)=1$, this recursion uniquely determines
$\mu_{n+1}(\mathfrak c)$ in terms of $\mu_0(\mathfrak c),\dots,\mu_n(\mathfrak c)$.
Moreover, each coefficient $A_{n,j}(\mathfrak c^2)$ is a polynomial in $\mathfrak c^2$. An induction on $n$, based on the triangular structure of the recursion, shows that,
for every $n\in\mathbb{N}$, $\mu_n(\mathfrak c)$ is a polynomial in $\mathfrak c^2$ of degree
at most $\lfloor n/2\rfloor$. In particular, the limit
\[
\lim_{\substack{\mathfrak c\to0\\ \mathfrak c\in U}}
\mu_n(\mathfrak c)
\]
exists and is finite. Define a linear functional $\mathbf u_0$ on $\mathcal P$ by prescribing its moments,
$\langle\mathbf u_0,x^n\rangle=\mu_n$ for all $n\in\mathbb{N}$, and extending linearly.
Then for every polynomial $p(x)=\sum_{k=0}^m a_kx^k$,
\[
\langle\mathbf u_{\mathfrak c},p\rangle
=
\sum_{k=0}^m a_k\mu_k(\mathfrak c)
\longrightarrow
\sum_{k=0}^m a_k\mu_k
=
\langle\mathbf u_0,p\rangle,
\]
as $\mathfrak c\to0$ within $U$. This is precisely convergence in the weak topology
$\sigma(\mathcal P',\mathcal P)$. Uniqueness of $\mathbf u_0$ follows from the fact that a
linear functional on $\mathcal P$ is uniquely determined by its values on monomials.
\end{proof}

\begin{proposition}\label{prop:limit-satisfies-continuous-equation}
Under the hypotheses of Lemma~\ref{lem:weak-limit-exists}, let
$\mathbf u_0\in\mathcal P'$ denote the weak limit obtained there.
Let $\mathrm D:\mathcal P\to\mathcal P$ be the usual derivative, and let $\boldsymbol{\mathrm D}:\mathcal P'\to\mathcal P'$ denote the transpose of $\mathrm D$,
understood in the sense of Definition~\cite[Definition~1.2]{CP25}, that is,
\[
\langle \boldsymbol{\mathrm D}\mathbf u,\,p\rangle
=
-\langle \mathbf u,\,\mathrm D p\rangle,
\quad p\in\mathcal P.
\]
Then $\mathbf u_0$ satisfies the functional equation
\[
\boldsymbol{\mathrm D}(\phi\,\mathbf u_0)=\psi\,\mathbf u_0 .
\]
\end{proposition}

\begin{proof}
Fix $p\in\mathcal P$. Pairing the functional equation in Lemma~\ref{lem:weak-limit-exists} against $p$ and using transposition yields
\[
\bigl\langle \mathbf u_{\mathfrak c},
\,\phi\,\mathrm D_{X}(p)+\psi\,\mathrm S_{X}(p)\bigr\rangle=0,
\quad \mathfrak c\in U.
\]
Set
\[
q_{\mathfrak c}=\phi\,\mathrm D_{X}(p)+\psi\,\mathrm S_{X}(p).
\]
Since $\phi$ has degree at most two and $\psi$ has degree at most one, the polynomial
$q_{\mathfrak c}$ has degree at most one more than that of $p$. Hence there exists
$m\in\mathbb N$, depending only on $p$, such that
$q_{\mathfrak c}\in\mathcal P_m$ for all $\mathfrak c\in U$.
By Remark~\ref{rem:DX-SX-limit}, one has
\[
\mathrm D_{X}(p)\to \mathrm D p,
\quad
\mathrm S_{X}(p)\to p
\]
in $\mathcal P_m$ as $\mathfrak c\to0$ within $U$. Consequently,
\[
q_{\mathfrak c}\to q=\phi\,\mathrm D p+\psi\,p
\]
in $\mathcal P_m$.
Writing
\[
q_{\mathfrak c}(x)=\sum_{k=0}^{m}\alpha_k(\mathfrak c)\,x^k,
\quad
q(x)=\sum_{k=0}^{m}\alpha_k\,x^k,
\]
the convergence in $\mathcal P_m$ implies $\alpha_k(\mathfrak c)\to\alpha_k$ for each
$k=0,\dots,m$. By Lemma~\ref{lem:weak-limit-exists},
\[
\mu_k(\mathfrak c)\to\langle \mathbf u_0,x^k\rangle
\] 
for each fixed $k$. Since the sum is finite, convergence of the coefficients together with convergence of
the moments implies convergence of the pairing. Therefore,
\[
\langle \mathbf u_{\mathfrak c},q_{\mathfrak c}\rangle
\longrightarrow
\langle \mathbf u_0,q\rangle.
\]
Since $\langle \mathbf u_{\mathfrak c},q_{\mathfrak c}\rangle=0$ for all $\mathfrak c\in U$,
passing to the limit yields
\[
\langle \mathbf u_0,\phi\,\mathrm D p+\psi\,p\rangle=0
\]
for all $p\in\mathcal P$. By transposition, the above identity yields the desired conclusion, which completes the
proof.
\end{proof}

Without loss of generality, since every family in the sense of
Theorem~\ref{thm:canonicalreps} is affinely equivalent to a canonical one,
the examples below are formulated for a specific representative case,
from which the classical families in the sense of Maroni can be
transparently recovered.

\begin{example}[Hermite case]\label{ex:1H-to-0H-maximal-U}
Fix $\mathfrak d\in\mathbb C$. For each $\mathfrak c\in\mathbb C^\times$ set
\(
X(s)=\mathfrak c s+\mathfrak d\). Consider the \(1\)-Hermite functional equation
\[
\boldsymbol{\mathrm D}_{X}(\mathbf u_{\mathfrak c})
=
\boldsymbol{\mathrm S}_{X}\bigl((-2x)\,\mathbf u_{\mathfrak c}\bigr).
\]
In this case $\phi(x)=1$ and $\psi(x)=-2x$, hence $a=b=e=0$, $c=1$, $d=-2$, and therefore
\[
d_n=-2,
\]
so the hypotheses of Lemma~\ref{lem:weak-limit-exists} are satisfied.
By Theorem~\ref{thm:linear-lattice}, non-regularity occurs precisely when
$\phi^{[n]}(0)=0$ for some $n\in\mathbb N$. Here
\[
\phi^{[n]}(x)=1-\frac{n}{2}\,\mathfrak c^{\,2},
\]
so $\phi^{[n]}(0)=0$ is equivalent to $\mathfrak c^{\,2}=2/n$.
Hence the set of non-regularity points is
\[
\mathcal N_{\mathcal H}
=
\left\{\pm \sqrt{\frac{2}{n}}:\ n\in\mathbb N^\times\right\},
\]
which is discrete\footnote{All figures in this and the following examples are obtained from explicit computations,
rather than illustrative representations. They were generated using
Wolfram~Mathematica~13.} and accumulates at $0$ with algebraic rate
$|\mathfrak c|\sim n^{-1/2}$ $($see Figure~\ref{fig:U-Hermite}$)$. 
\begin{figure}[t]
  \centering
  \includegraphics[width=0.4\textwidth]{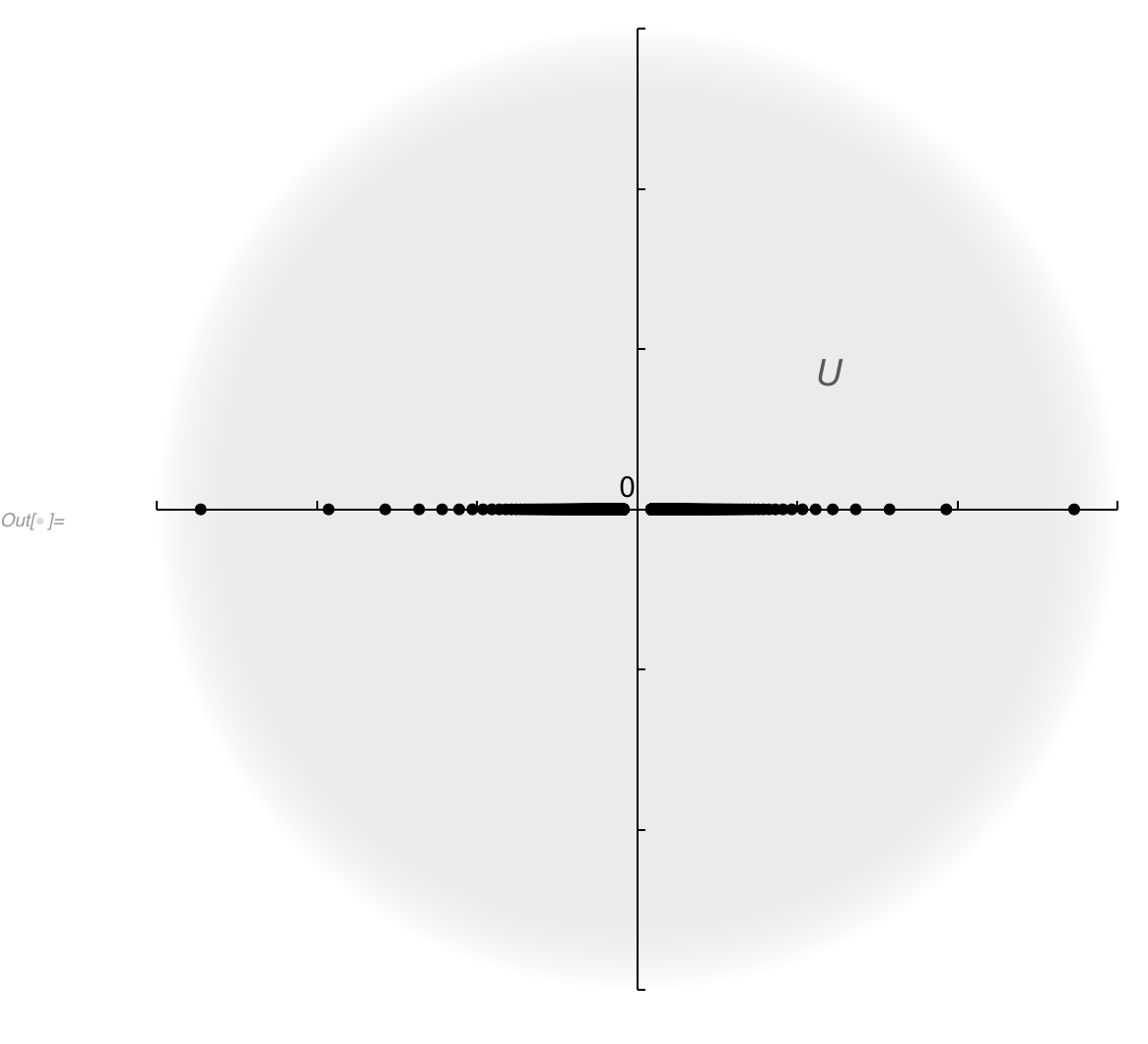}
  \caption{Admissible domains $U$ (shaded region)  in the $\mathfrak c$-plane, along with the corresponding non-regularity points (black disks), for the $1$-Hermite functionals with $n=1,\dots,1000$.}
  \label{fig:U-Hermite}
\end{figure}
Define
\[
U=\mathbb C^\times\setminus \mathcal N_{\mathcal H}.
\]
Then $0\in\overline U$. Moreover, by the above characterisation of
$\mathcal N_{\mathcal H}$, for every $\mathfrak c\in U$ 
Theorem~\ref{thm:linear-lattice} shows that $\mathbf u_{\mathfrak c}$ is regular. Choose for each $\mathfrak c\in U$ a nonzero solution
$\mathbf u_{\mathfrak c}\in\mathcal P'$
of the canonical equation and normalise it by
\[
\langle \mathbf u_{\mathfrak c},1\rangle=1,
\quad \mathfrak c\in U.
\]
Lemma~\ref{lem:weak-limit-exists} then yields a unique
$\mathbf u_0 \in\mathcal P'$ such that
\[
\mathbf u_{\mathfrak c}\longrightarrow
\mathbf u_0
\]
in $\sigma(\mathcal P',\mathcal P)$ as $\mathfrak c$ tends to $0$ within $U$.
Moreover, Proposition~\ref{prop:limit-satisfies-continuous-equation} gives
\[
\boldsymbol{\mathrm D}(\mathbf u_0)
=
(-2x)\,\mathbf u_0,
\]
so that $\mathbf u_0$ satisfies the canonical Hermite functional equation in~Table~\ref{table}.
\end{example}

\begin{example}[Laguerre case]\label{ex:1L-to-0L-maximal-U}
Fix $\alpha\in\mathbb C$ and $\mathfrak d\in\mathbb C$. For each $\mathfrak c\in\mathbb C^\times$ set
\(
X(s)=\mathfrak c s+\mathfrak d\). Consider the \(1\)-Laguerre functional equation
\[
\boldsymbol{\mathrm{D}}_{X}\bigl(x\,\mathbf u_{\mathfrak c}\bigr)
=
\boldsymbol{\mathrm{S}}_{X}\bigl((1+\alpha-x)\,\mathbf u_{\mathfrak c}\bigr).
\]
In this case $\phi(x)=x$ and $\psi(x)=1+\alpha-x$, hence $a=0$, $b=1$, $c=0$, $d=-1$, $e=\alpha+1$,
and therefore
\[
d_n=-1,
\]
so the hypotheses of Lemma~\ref{lem:weak-limit-exists} are satisfied. By Theorem~\ref{thm:linear-lattice}, non-regularity occurs precisely when
\[
\phi^{[n]}\!\left(-\frac{e_n}{d_{2n}}\right)=0,
\]
for some $n\in\mathbb N$. Here
\[
e_n=n+\alpha+1,
\quad
d_{2n}=-1,
\quad
\phi^{[n]}(x)=x-\frac{n}{4}\,\mathfrak c^{\,2},
\]
so that
\[
\phi^{[n]}\!\left(-\frac{e_n}{d_{2n}}\right)
=
(n+\alpha+1)-\frac{n}{4}\,\mathfrak c^{\,2}.
\]
Hence non-regularity occurs if and only if
\[
\mathfrak c^{\,2}
=
4\,\frac{n+\alpha+1}{n}.
\]
\begin{figure}[t]
  \centering
  \includegraphics[width=0.8\textwidth]{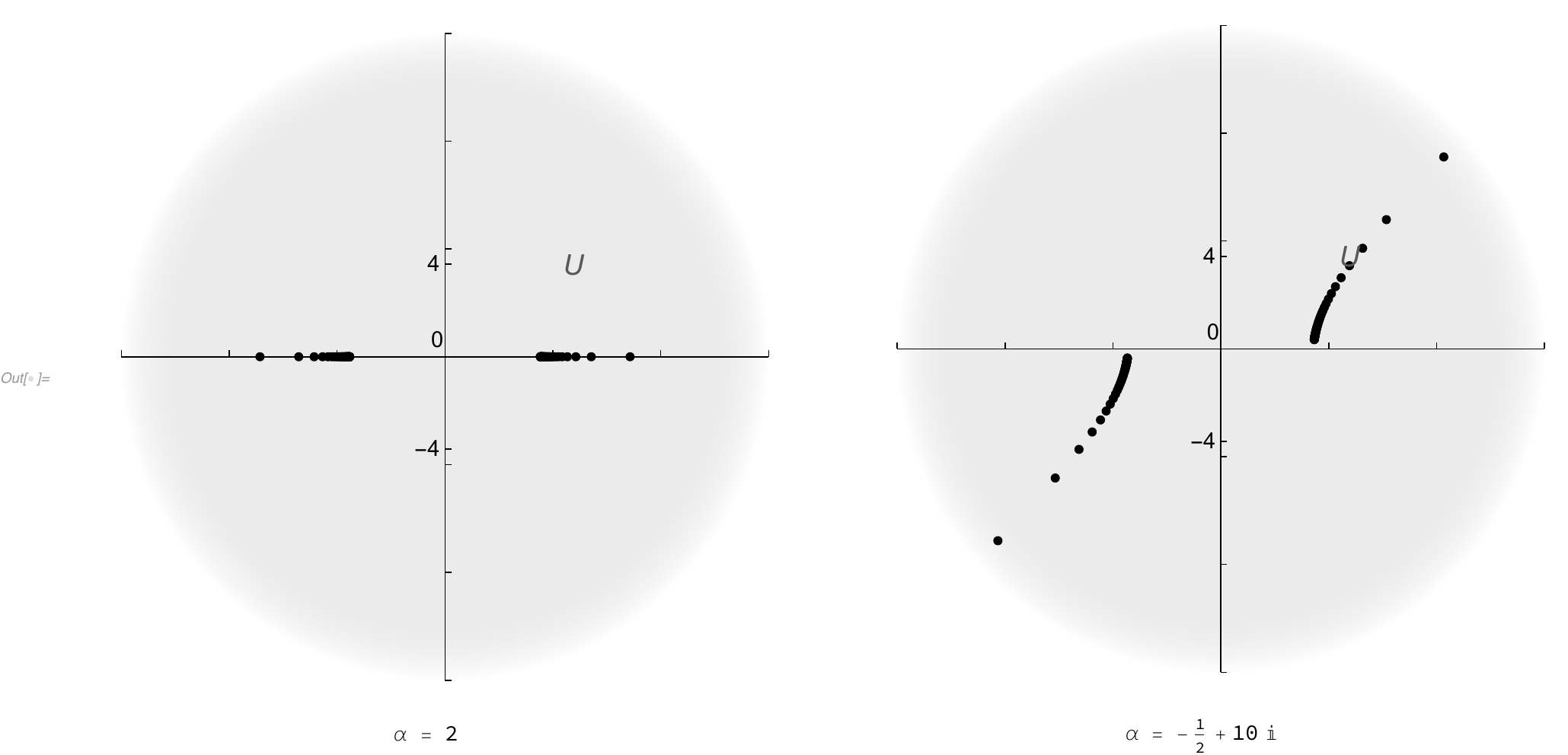}
  \caption{Admissible domains $U$ in the $\mathfrak c$-plane (shaded region), along with the corresponding non-regularity points (black disks), for the $1$-Laguerre family with $n=1,\dots,50$.}
  \label{fig:U-Laguerre}
\end{figure}
Accordingly, the set of non-regularity points is
\[
\mathcal N_{\mathcal L}^{(\alpha)}
=\left\{\pm\,2\,\sqrt{\dfrac{n+\alpha+1}{n}}:\ n\in\mathbb N^\times\right\}.
\]
Define
\[
U=\mathbb C^\times\setminus \mathcal N_{\mathcal L}^{(\alpha)}.
\]
Then $0\in\overline U$. Moreover, by the above characterisation of
$\mathcal N_{\mathcal L}^{(\alpha)}$, for every $\mathfrak c\in U$ one has
\[
\phi^{[n]}\!\left(-\frac{e_n}{d_{2n}}\right)\neq0,
\]
and therefore Theorem~\ref{thm:linear-lattice} shows that $\mathbf u_{\mathfrak c}$ is regular. Figure~\ref{fig:U-Laguerre} shows that non‑regularity does not accumulate at $0$ but
instead converges to finite nonzero limits, whose geometry depends on $\alpha$. Choose for each $\mathfrak c\in U$ a nonzero solution
$\mathbf u_{\mathfrak c}\in\mathcal P'$
of the canonical equation and normalise it by
\[
\langle \mathbf u_{\mathfrak c},1\rangle=1,
\quad \mathfrak c\in U.
\]
Lemma~\ref{lem:weak-limit-exists} then yields a unique
$\mathbf u_0 \in\mathcal P'$ such that
\[
\mathbf u_{\mathfrak c}\longrightarrow
\mathbf u_0
\]
in $\sigma(\mathcal P',\mathcal P)$ as $\mathfrak c$ tends to $0$ within $U$.
Moreover, Proposition~\ref{prop:limit-satisfies-continuous-equation} gives
\[
\boldsymbol{\mathrm D}\bigl(x\,\mathbf u_0\bigr)
=
(-x+\alpha+1)\,\mathbf u_0,
\]
so that $\mathbf u_0$ satisfies the canonical Laguerre functional equation in~Table~\ref{table}.
\end{example}

\begin{example}[Bessel case]\label{ex:1B-to-0B}
Fix $\alpha\in\mathbb C$ and $\mathfrak d\in\mathbb C$. For each
$\mathfrak c\in\mathbb C^\times$ set
\(
X(s)=\mathfrak c s+\mathfrak d\).
Consider the \(1\)-Bessel functional equation
\[
\boldsymbol{\mathrm{D}}_{X}\bigl(x^2\,\mathbf u_{\mathfrak c}\bigr)
=
\boldsymbol{\mathrm{S}}_{X}\bigl(((\alpha+2)x+2)\,\mathbf u_{\mathfrak c}\bigr).
\]
In this case $\phi(x)=x^2$ and $\psi(x)=(\alpha+2)x+2$, hence
\[
a=1,\quad b=0,\quad c=0,\quad d=\alpha+2,\quad e=2,
\]
and therefore
\[
d_n=n+\alpha+2,
\quad
e_n=2.
\]
Assume $\alpha\notin -2-\mathbb N$, so that $d_n\neq0$ for all $n\in\mathbb N$.
The hypotheses of Lemma~\ref{lem:weak-limit-exists} are therefore satisfied.
In the notation of Theorem~\ref{thm:linear-lattice}, we have
\[
\phi^{[n]}(x)
=
x^2+\frac{n}{4}\,\mathfrak c^{\,2}(n+\alpha+2),
\quad
-\frac{e_n}{d_{2n}}
=
-\frac{2}{2n+\alpha+2},
\]
and hence
\[
\phi^{[n]}\!\left(-\frac{e_n}{d_{2n}}\right)
=
\frac{4}{(2n+\alpha+2)^2}
+
\frac{n}{4}\,\mathfrak c^{\,2}(n+\alpha+2).
\]
Hence non-regularity occurs if and only if
\[
\mathfrak c^{\,2}
=
-\frac{16}{n\,(n+\alpha+2)\,(2n+\alpha+2)^2}.
\]
\begin{figure}[t]
  \centering
  \includegraphics[width=0.8\textwidth]{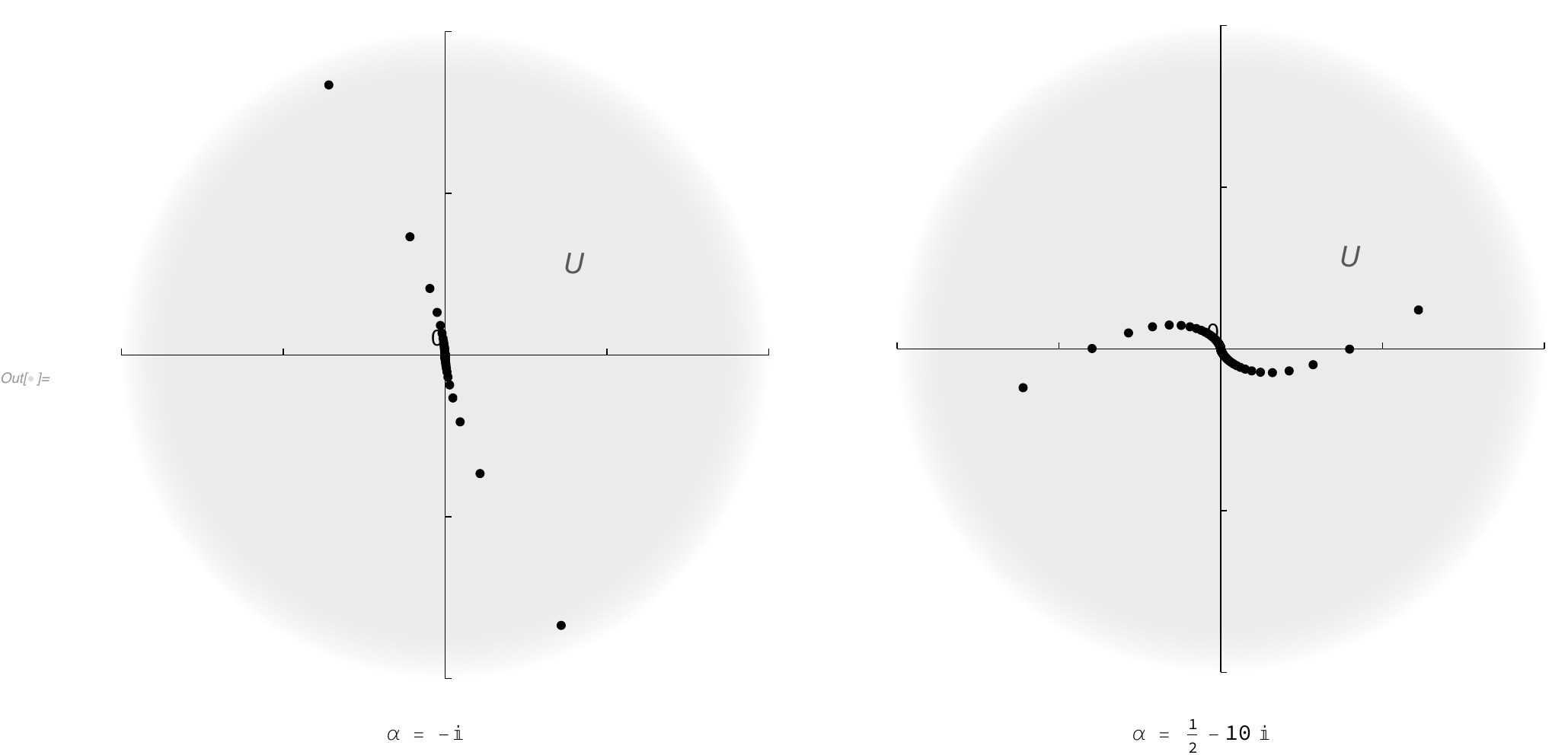}
  \caption{Admissible domains $U$ in the $\mathfrak c$-plane (shaded region), together
  with the corresponding non-regularity points (black disks), for the $1$-Bessel family
  with $n=1,\dots,50$.}
  \label{fig:U-bessel}
\end{figure}
Accordingly, the set of non-regularity points is
\[
\mathcal N_{\mathcal B}^{(\alpha)}
=
\left\{
\pm\,\frac{4 i}{\sqrt{\,n\,(n+\alpha+2)\,}\,(2n+\alpha+2)}
:\ n\in\mathbb N^\times
\right\},
\]
which is discrete and accumulates at $0$ with algebraic rate
$|\mathfrak c|\sim n^{-2}$ $($see Figure~\ref{fig:U-bessel}$)$.
Define
\[
U=\mathbb C^\times\setminus \mathcal N_{\mathcal B}^{(\alpha)}.
\]
Then $0\in\overline U$. Moreover, by the above characterisation of
$\mathcal N_{\mathcal B}^{(\alpha)}$, for every $\mathfrak c\in U$ one has
\[
\phi^{[n]}\!\left(-\frac{e_n}{d_{2n}}\right)\neq0,
\]
and therefore Theorem~\ref{thm:linear-lattice} shows that $\mathbf u_{\mathfrak c}$ is regular. Choose for each $\mathfrak c\in U$ a nonzero solution
$\mathbf u_{\mathfrak c}\in\mathcal P'$
of the canonical equation and normalise it by
\[
\langle \mathbf u_{\mathfrak c},1\rangle=1,
\quad \mathfrak c\in U.
\]
Lemma~\ref{lem:weak-limit-exists} then yields a unique
$\mathbf u_0\in\mathcal P'$ such that
\[
\mathbf u_{\mathfrak c}\longrightarrow
\mathbf u_0
\]
in $\sigma(\mathcal P',\mathcal P)$ as $\mathfrak c$ tends to $0$ within $U$.
Moreover, Proposition~\ref{prop:limit-satisfies-continuous-equation} gives
\[
\boldsymbol{\mathrm D}\bigl(x^2\,\mathbf u_0\bigr)
=
((\alpha+2)x+2)\,\mathbf u_0,
\]
so that $\mathbf u_0$ satisfies the canonical Bessel functional equation in~Table~\ref{table}.
\end{example}

\begin{example}[Jacobi case]\label{ex:1J-to-0J}
Fix $\alpha,\beta\in\mathbb C$ and $\mathfrak d\in\mathbb C$. For each
$\mathfrak c\in\mathbb C^\times$ set
\(
X(s)=\mathfrak c s+\mathfrak d\).
Consider the \(1\)-Jacobi functional equation
\[
\boldsymbol{\mathrm{D}}_{X}\bigl((1-x^2)\,\mathbf u_{\mathfrak c}\bigr)
=
\boldsymbol{\mathrm{S}}_{X}\bigl((\beta-\alpha-(\alpha+\beta+2)x)\,\mathbf u_{\mathfrak c}\bigr).
\]
In this case $\phi(x)=1-x^2$ and $\psi(x)=\beta-\alpha-(\alpha+\beta+2)x$, hence
\[
a=-1,\quad b=0,\quad d=-(\alpha+\beta+2),\quad e=\beta-\alpha,
\]
and therefore
\[
d_n=-(n+\alpha+\beta+2),
\quad
e_n=\beta-\alpha.
\]
Assume $-\alpha-\beta-2\notin\mathbb N$, so that $d_n\neq0$ for all $n\in\mathbb N$.
The hypotheses of Lemma~\ref{lem:weak-limit-exists} are therefore satisfied.
In the notation of Theorem~\ref{thm:linear-lattice}, we have
\[
\phi^{[n]}(x)
=
1-x^2-\frac{n}{4}\,(n+\alpha+\beta+2)\,\mathfrak c^{\,2},
\quad
-\frac{e_n}{d_{2n}}
=
\frac{\beta-\alpha}{2n+\alpha+\beta+2},
\]
and hence
\[
\phi^{[n]}\!\left(-\frac{e_n}{d_{2n}}\right)
=
1-
\left(\frac{\beta-\alpha}{2n+\alpha+\beta+2}\right)^2
-
\frac{n}{4}\,(n+\alpha+\beta+2)\,\mathfrak c^{\,2}.
\]
Hence non-regularity occurs if and only if
\[
\mathfrak c^{\,2}
=
-\frac{4}{n\,(n+\alpha+\beta+2)}
\left(
\left(\frac{\beta-\alpha}{2n+\alpha+\beta+2}\right)^2-1
\right).
\]

\begin{figure}[t]
  \centering
  \includegraphics[width=0.8\textwidth]{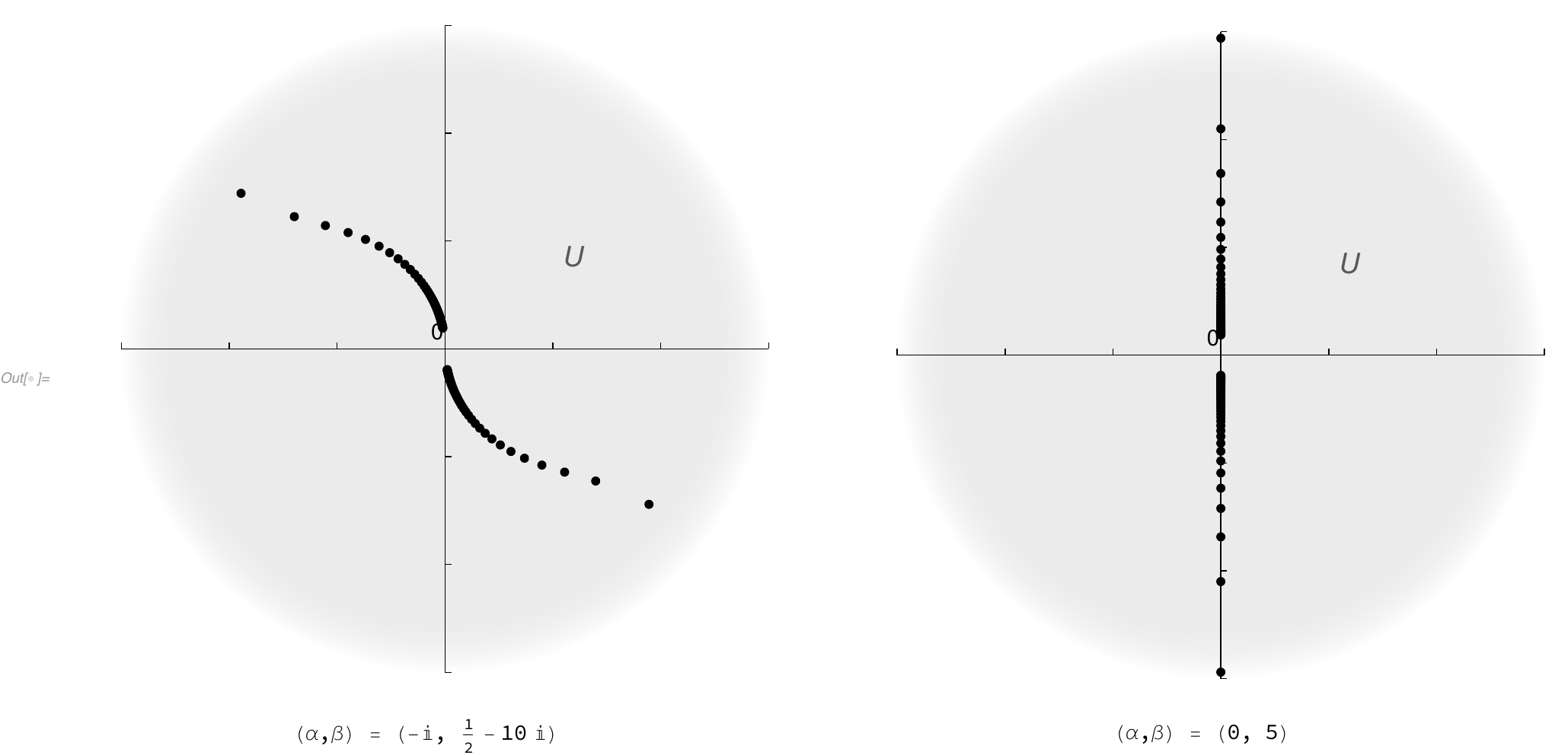}
  \caption{Admissible domains $U$ in the $\mathfrak c$-plane (shaded region), along with
  the corresponding non-regularity points (black disks), for the $1$-Jacobi family with
  $n=1,\dots,20$.}
  \label{fig:U-jacobi}
\end{figure}
Accordingly, the set of non-regularity points is
\[
\mathcal N_{\mathcal J}^{(\alpha,\beta)}
=
\left\{
\pm\,\frac{2 i}{\sqrt{\,n\,(n+\alpha+\beta+2)\,}}
\sqrt{
1-
\left(\frac{\beta-\alpha}{2n+\alpha+\beta+2}\right)^2
}
:\ n\in\mathbb N^\times
\right\},
\]
which is discrete and accumulates at $0$ with algebraic rate
$|\mathfrak c|\sim n^{-1}$ $($see Figure~\ref{fig:U-jacobi}$)$.
Define
\[
U=\mathbb C^\times\setminus \mathcal N_{\mathcal J}^{(\alpha,\beta)}.
\]
Then $0\in\overline U$. Moreover, by the above characterisation of
$\mathcal N_{\mathcal J}^{(\alpha,\beta)}$, for every $\mathfrak c\in U$ one has
\[
\phi^{[n]}\!\left(-\frac{e_n}{d_{2n}}\right)\neq0,
\]
and therefore Theorem~\ref{thm:linear-lattice} shows that $\mathbf u_{\mathfrak c}$ is regular.
Choose for each $\mathfrak c\in U$ a nonzero solution
$\mathbf u_{\mathfrak c}\in\mathcal P'$
of the canonical equation and normalise it by
\[
\langle \mathbf u_{\mathfrak c},1\rangle=1,
\quad \mathfrak c\in U.
\]
Lemma~\ref{lem:weak-limit-exists} then yields a unique
$\mathbf u_0 \in\mathcal P'$ such that
\[
\mathbf u_{\mathfrak c}\longrightarrow
\mathbf u_0
\]
in $\sigma(\mathcal P',\mathcal P)$ as $\mathfrak c$ tends to $0$ within $U$.
Moreover, Proposition~\ref{prop:limit-satisfies-continuous-equation} gives
\[
\boldsymbol{\mathrm D}\bigl((1-x^2)\,\mathbf u_0\bigr)
=
(-(\alpha+\beta+2)x-\alpha+\beta)\,\mathbf u_0,
\]
so that $\mathbf u_0$ satisfies the canonical Jacobi functional equation in~Table~\ref{table}.
\end{example}

The previous discussion shows that Theorem~\ref{thm:linear-lattice} is compatible with the $0$-classical framework in the limit $\mathfrak c\to0$, interpreted via weak convergence in $\sigma(\mathcal P',\mathcal P)$ along admissible domains, thereby placing the lattice theory within its natural limiting context without explicitly addressing the $0$-classical regularity theorem.

\section{The NIST DLMF and Koekoek--Lesky--Swarttouw classifications}\label{nist}

The orthogonal polynomial sequences solving the difference equation
\eqref{eq:KLSdiff-eq}, with real parameters as in
\cite[Theorem~5.2]{KLS10} and with complex parameters as in
\cite[Theorem~6.1]{KLS10}, which together include all those presented in
the \emph{NIST DLMF}, arise as particular instances of the theory developed here.
Indeed, by Observation~\ref{prop:KLS10-functional}, the difference equation appearing in the statement of \cite[Theorems~5.2 and~6.1]{KLS10} can be rewritten as the functional equation
\[
  \boldsymbol{\nabla}\!\bigl(\phi\,\mathbf u\bigr)=\psi\,\mathbf u,
  \quad
  \phi(x)=e(x-1)^2+2f(x-1)+g,
  \quad
  \psi(x)=2\epsilon(x-1)+\gamma.
\]
On the linear lattice \(X_0(s)=s\), Observation~\ref{prop:KLS10-functional} yields
\begin{equation}\label{eq:DS_general}
  \boldsymbol{\mathrm D}_{X_0}\!\bigl(\Phi\,\mathbf u\bigr)
  =
  \boldsymbol{\mathrm S}_{X_0}\!\bigl(\Psi\,\mathbf u\bigr),
  \quad
  \Phi=2\phi-\psi,
  \quad
  \Psi=2\psi,
\end{equation}
that is,
\begin{equation}\label{eq:PhiPsi_KLS10}
  \Phi(x)=2e(x-1)^2+(4f-2\epsilon)(x-1)+(2g-\gamma),
  \quad
  \Psi(x)=4\epsilon(x-1)+2\gamma.
\end{equation}
In particular, the functional $\mathbf u$ is $1$-classical in the sense of
Definition~\ref{defclass}. Consequently, Theorem~\ref{thm:canonicalreps},
applied as in Examples~\ref{K2} and \ref{ex:GMS95-Table1-to-canonical-integrated},
identifies the corresponding family together with its parameters
directly from the pair $(\Phi,\Psi)$. It is worth emphasizing that, from the present functional viewpoint,
the distinction between real parameters \cite[Theorem~5.2]{KLS10}
and complex parameters \cite[Theorem~6.1]{KLS10} is artificial.
Omitting all calculations and avoiding the names collected in \cite{KLS10},
we place all these families among those recorded in Table~\ref{table:repr} relying solely on the case-by-case
subdivision of \cite[Theorems~5.2 and~6.1]{KLS10}. Specifically, following this subdivision, we adopt the notation
\[
  P_n^{(A)}, \quad Q_n^{(B)},
\]
for the corresponding monic orthogonal polynomial families in the real and complex
settings, respectively.
Here the superscripts \(A\) and \(B\) serve merely as labels for the cases considered
in \cite{KLS10}: accordingly, \(A\) ranges over the real-case labels appearing in
\cite[Theorem~5.2]{KLS10}, while \(B\) ranges over the complex-case labels appearing in
\cite[Theorem~6.1]{KLS10}.

\subsection{1-Hermite case}
This corresponds to the case \(e=0\) and \(4f=2\epsilon\) (equivalently, \(\Phi\) is constant). It comprises Cases~IIb1 and IIb2 (where \(2f=1\) forces \(\epsilon=1\)) in \cite[Theorem~5.2]{KLS10}.
In this case, one has the identifications
\begin{align*}
  P_n^{(\mathrm{IIb1})}(x)
  &= H_n^{(-4/N)}\!\left(x-\tfrac{N}{2}\right),
&& n=0,1,\dots,N,\\[7pt]
  P_n^{(\mathrm{IIb2})}(x)
  &= H_n^{(4/(\gamma-2))}\!\left(x+\tfrac{\gamma-2}{2}\right),
&& n=0,1,\dots,N,
\end{align*}
with the parameter constraints as stated in
\cite[Theorem~5.2]{KLS10}, which we make explicit here, and only in this example, exactly as they appear there, for the sole purpose of illustrating them for the reader.
{Case IIb1:} $2f=1$, $\epsilon=1$, $g=\gamma-2\epsilon+1$, and $\gamma-2\epsilon=-N$.
{Case IIb2:} $2f=1$, $\epsilon=1$, $g=\gamma-2\epsilon+1$, and $-N<\gamma-2\epsilon<-N+1$.

\subsection{1-Laguerre case}
This corresponds to the case \(e=0\) and \(4f\neq2\epsilon\) (equivalently, \(\Phi\) has degree $1$).
It comprises Cases~I, IIa1, IIa2, IIa3, and Cases~IIb1 and IIb2 with \(\epsilon\neq1\) in \cite[Theorem~5.2]{KLS10}
and Case~II in \cite[Theorem~6.1]{KLS10}.
In this case, one has the identifications
\begin{align*}
  P_n^{(I)}(x)
  &=
  L_n^{(-2,\,-2/\epsilon)}\!\left(x-\tfrac{1}{2\epsilon}\right),
  && n\in\mathbb{N},\\[7pt]
  P_n^{(\mathrm{IIa1})}(x)
  &=
  L_n^{\left(\frac{2\epsilon}{\epsilon-1},\,\frac{\gamma-2(1+\gamma-2\epsilon)\epsilon}{(\epsilon-1)^2}\right)}\!\left(x+\tfrac{\gamma-2\epsilon}{2-2\epsilon}\right),
  && n\in\mathbb{N},\\[7pt]
  P_n^{(\mathrm{IIa2})}(x)
  &=
  L_n^{\left(\frac{2\epsilon}{\epsilon-1},\,\frac{\gamma-2\epsilon}{(\epsilon-1)^2}\right)}\!\left(x-\tfrac{\gamma-2\epsilon}{2-2\epsilon}\right),
  && n\in\mathbb{N},\\[7pt]
  P_n^{(\mathrm{IIa3})}(x)
  &=
  L_n^{(2,\,4(\gamma-1))}(x+1-\gamma),
  && n\in\mathbb{N},\\[7pt]
  P_n^{(\mathrm{IIb1})}(x)
  &=
  L_n^{(2,\, 4(N-1+\gamma)}\!\left(x+1-2N-\gamma \right),
  && n=0,1,\dots,N,\\[7pt]
  P_n^{(\mathrm{IIb2})}(x)
  &=
  L_n^{(2,\,8\epsilon-4)}\!\left(x+1-4\epsilon+\gamma \right),
  && n=0,1,\dots,N,\\[7pt]
  Q_n^{(\mathrm{II})}(x)
  &=L_n^{(\alpha,\beta)}(x+\xi),
&& n\in\mathbb{N},
\end{align*}
with
\[
  \xi=\frac{2g+2\epsilon-4f-\gamma}{4f-2\epsilon},
  \quad
  \alpha=\frac{4\epsilon}{4f-2\epsilon},
  \quad
  \beta=\frac{(4f-2\epsilon)(2\gamma-4\epsilon)-4\epsilon(2g+2\epsilon-4f-\gamma)}{(4f-2\epsilon)^2},
\]
with the parameter constraints as stated in \cite[Theorems~5.2 and 6.1]{KLS10}. Although in Cases~IIb2 and~IIIa2 of \cite[Theorem~5.2]{KLS10} the orthogonality is written
in terms of a Barnes-type contour integral, this does not indicate a non--positive
functional.
Under the parameter restrictions stated in \cite[Section~5.3]{KLS10}, the corresponding
functional is positive-definite in the usual sense: the contour representation can be
reduced, by residue calculus, to a finite discrete orthogonality relation with strictly
positive weights supported on a finite set.
Thus, the use of a complex contour in the representation is merely an analytic device
and should not be confused with the genuinely non--positive situations arising, for
instance, in the case of the Bessel polynomials.

\subsection{1-Jacobi and \(1\)-Bessel cases}

This corresponds to the case \(e\neq0\) (equivalently, \(\Phi\) has degree $2$). It comprises Cases IIIa1, IIIa2, IIIb1, IIIb2, IIIb3, IIIc in \cite[Theorem~5.2]{KLS10} and Cases IIIa and IIIb in \cite[Theorem~6.1]{KLS10}.
In this case, one has the identifications
\begin{align*}
  P_n^{(A)}(x)
  &= J_n^{(\alpha,\beta,\gamma)}(x+\xi),
  \qquad n\in\mathbb{N},\\[7pt]
  Q_n^{(B)}(x)
  &= J_n^{(\alpha,\beta,\gamma)}(x+\xi),
  \qquad n\in\mathbb{N}.
\end{align*}
where
\[
  \alpha=2\epsilon,
  \quad
  \beta=\gamma-2f\epsilon+\epsilon^{2},
  \quad
  \gamma=\frac{(4f-2\epsilon)^{2}}{16}-g+\frac{\gamma}{2},
  \quad
  \xi=\frac{4f-2\epsilon-4}{4},
\]
with the parameter constraints as stated in \cite[Theorems~5.2 and 6.1]{KLS10}.

Finally, note that the family \(Q_n^{(\mathrm{II})}\) is usually referred to as the
Meixner-Pollaczek polynomials, whereas from our viewpoint it is simply a
\(1\)-Laguerre family, no more and no less than \(P_n^{(\mathrm{IIa2})}\), commonly called
Meixner, or \(P_n^{(\mathrm{IIb2})}\), often labeled as Krawtchouk.
In fact, all these families belong to the same equivalence class as defined in
Definition~\ref{defequi}, and their mutual relations under affine changes of the variable
are governed by Theorem~\ref{thm:canonicalreps}, within the freedom allowed by the
regularity assumptions discussed there.
More importantly, all of them lead, through a limiting procedure explained in the
preceding section, to the corresponding Laguerre family already appearing in
Bochner’s 1929 work.

\section{The explicit representations of the functional}\label{representacion}
The explicit representation of the functionals was deliberately avoided throughout this work, since it neither occupies nor should occupy a central place in the classification of orthogonal polynomials based on their algebraic properties. This choice, however, should not be interpreted as a limitation of the present approach, which is fully capable of producing such representations whenever they are required.

\begin{proposition}\label{lem:cnzero-general}
Let $x_0\in\mathbb{C}$ and let $(c_n)_{n\in I}$ be complex numbers, where
$I\subset\mathbb{Z}$ is an interval (finite or infinite, possibly unbounded on either side).
Assume that for every polynomial $p$,
\begin{equation}\label{eq:weak-eval-general}
\sum_{n\in I} c_n\,p(x_0+n)=0,
\end{equation}
and that there exists $R>1$ such that
\begin{equation}\label{eq:exp-decay-general}
\sum_{n\in I} |c_n|\,R^{|n|}<\infty.
\end{equation}
Then $c_n=0$ for all $n\in I$.
\end{proposition}

\begin{proof}
For $m\in\mathbb{N}$ consider $p_m(z)=(z-x_0)^m$. Then $p_m(x_0+n)=n^m$, and \eqref{eq:weak-eval-general} yields
\begin{equation}\label{eq:moments-general}
\sum_{n\in I} c_n\,n^m=0.
\end{equation}
Fix $\zeta\in(0,\log R)$ and define
\[
F(t)=\sum_{n\in I}c_n\,\ex^{nt},\quad t\in\mathbb{C},\quad |\Re t|<\log R.
\]
By \eqref{eq:exp-decay-general},
\begin{align*}
\sum_{n\in I}\bigl|c_n\,\ex^{nt}\bigr|
&=\sum_{n\in I}|c_n|\,\ex^{n\Re t}\\[7pt]
&\le \sum_{n\in I}|c_n|\,\ex^{|n||\Re t|}
\le \sum_{n\in I}|c_n|\,\ex^{a|n|}
<\infty
\end{align*}
whenever $|\Re t|\le \zeta$. Hence $F$ converges absolutely and uniformly on each closed vertical strip $\{t:\,|\Re t|\le \zeta\}$ and is therefore holomorphic on the open strip $\{t:\,|\Re t|<\log R\}$. Uniform convergence allows termwise differentiation on that strip, giving
\[
F^{(m)}(t)=\sum_{n\in I} c_n n^m\,\ex^{nt}.
\]
Evaluating at $t=0$ and using \eqref{eq:moments-general} we obtain $F^{(m)}(0)=0$ for all $m$, whence $F\equiv 0$ on $\{|\Re t|<\log R\}$. Let $z=\ex^{t}$. Then $R^{-1}<|z|<R$ when $|\Re t|<\log R$, and
\[
G(z)=\sum_{n\in I} c_n\,z^{n}
\]
is a holomorphic function on its natural domain of convergence (a disc, an exterior domain, or an annulus according to the ends of $I$). On the annulus $R^{-1}<|z|<R$ we have $G(z)=F(t)=0$. By the identity theorem, $G$ vanishes identically on its domain, and therefore all its Laurent coefficients vanish: $c_n=0$ for every $n\in I$.
\end{proof}

\begin{coro}\label{cor:cnzero-strong}
Let $\xi\in\mathbb C$ and let $(c_n)_{n\ge -1}$ be complex numbers. Assume that
for every polynomial $p$,
\begin{equation}\label{eq:weak-eval-cor}
\sum_{n=-1}^\infty c_n\,p\!\left(\xi+n+\tfrac12\right)=0,
\end{equation}
and that there exists $R>1$ such that
\begin{equation}\label{eq:exp-decay-cor}
\sum_{n=0}^\infty |c_n|\,R^{n}<\infty.
\end{equation}
Then $c_n=0$ for all $n\in\{-1\}\cup\mathbb N$.
\end{coro}

\begin{proof}
This is a direct consequence of Proposition~\ref{lem:cnzero-general} applied to the
arithmetic progression $x_n=\xi+n+\tfrac12$, with index set
$I=\{-1,0,1,2,\dots\}$. The exponential summability condition
\eqref{eq:exp-decay-cor} ensures the validity of the hypotheses of
Lemma~\ref{lem:cnzero-general}, and the conclusion follows immediately.
\end{proof}

\begin{definition}
Let $\alpha\in\mathbb C$. The point evaluation at $\alpha$ defines a
continuous linear functional
\[
\boldsymbol{\delta}_\alpha\in\mathcal P',\qquad
\langle \boldsymbol{\delta}_\alpha,p\rangle=p(\alpha),
\quad p\in\mathcal P.
\]
We refer to $\boldsymbol{\delta}_\alpha$ as an \emph{atom} at $\alpha$.
\end{definition}

Suppose that $\mathbf u\in\mathcal P'$ satisfies
\begin{equation}\label{eq:DS-X0}
\boldsymbol{\mathrm D}_{X_0}(\Phi\,\mathbf u)
=
\boldsymbol{\mathrm S}_{X_0}(\Psi\,\mathbf u),
\qquad X_0(s)=s,
\end{equation}
with $(\Phi,\Psi)$ a canonical pair from Table~\ref{table:repr}.
We seek solutions whose support lies in one arithmetic progression and
consider
\begin{equation}\label{eq:atomic-ansatz}
\mathbf u=\sum_{k=0}^\infty \rho_k\,\boldsymbol{\delta}_{\xi+k},
\quad \xi\in\mathbb C,\quad \rho_0\neq0.
\end{equation}
The series in \eqref{eq:atomic-ansatz} defines an element of $\mathcal P'$ provided
that for every $m\in\mathbb N$,
\[
\sum_{k=0}^\infty |\rho_k|\,(1+k)^m<\infty,
\]
since then $|p(\xi+k)|\le C_m(1+k)^m$ for $\deg p\le m$, and
$\langle \mathbf u,p\rangle=\sum_{k=0}^\infty \rho_k\,p(\xi+k)$ converges absolutely. Using $p\,\boldsymbol{\delta}_a=p(a)\,\boldsymbol{\delta}_a$ and
\[
\boldsymbol{\mathrm D}_{X_0}
=\boldsymbol{\tau}_{-1/2}-\boldsymbol{\tau}_{1/2},\quad
2\,\boldsymbol{\mathrm S}_{X_0}
=\bigl(\boldsymbol{\tau}_{-1/2}+\boldsymbol{\tau}_{1/2}\bigr),
\]
we obtain
\begin{align*}
\boldsymbol{\mathrm D}_{X_0}(\Phi\,\mathbf u)
&=\sum_{k=0}^\infty\Phi(\xi+k)\rho_k\,\boldsymbol{\delta}_{\xi+k-\tfrac12}
-\sum_{k=0}^\infty\Phi(\xi+k)\rho_k\,\boldsymbol{\delta}_{\xi+k+\tfrac12},\\[4pt]
\boldsymbol{\mathrm S}_{X_0}(\Psi\,\mathbf u)
&=\frac12\sum_{k=0}^\infty\Psi(\xi+k)\rho_k\,\boldsymbol{\delta}_{\xi+k-\tfrac12}
+\frac12\sum_{k=0}^\infty\Psi(\xi+k)\rho_k\,\boldsymbol{\delta}_{\xi+k+\tfrac12}.
\end{align*}
Thus both sides are supported on
$\{\boldsymbol{\delta}_{\xi+n+\tfrac12}:n\in\{-1\}\cup\mathbb N\}$, and their
difference can be written as
\begin{equation}\label{eq:coef-sum}
\boldsymbol{\mathrm D}_{X_0}(\Phi\,\mathbf u)
-
\boldsymbol{\mathrm S}_{X_0}(\Psi\,\mathbf u)
=
\sum_{n=-1}^\infty c_n\,\boldsymbol{\delta}_{\xi+n+\tfrac12},
\end{equation}
where
\begin{align*}
c_{-1}&=\;\frac12\bigl(2\Phi(\xi)-\Psi(\xi)\bigr)\rho_0,\\[7pt]
c_n
   &=\frac12\Bigl((2\Phi(\xi+n+1)-\Psi(\xi+n+1))\rho_{n+1}
                   -(2\Phi(\xi+n)+\Psi(\xi+n))\rho_n\Bigr),
\quad n\in\mathbb N.
\end{align*}
If \eqref{eq:DS-X0} holds, then the left-hand side of \eqref{eq:coef-sum} vanishes
in $\mathcal P'$, and pairing against $p\in\mathcal P$ gives
\begin{equation}\label{eq:weak-eval}
\sum_{n=-1}^\infty c_n\,p\!\left(\xi+n+\tfrac12\right)=0.
\end{equation}
Assume in addition there exists $R>1$ with
$\sum_{n=0}^\infty |c_n|\,R^n<\infty$.
By Corollary~\ref{cor:cnzero-strong}, we conclude $c_n=0$ for all
$n\in\{-1\}\cup\mathbb N$, equivalently
\begin{align}
\Psi(\xi)&=2\,\Phi(\xi),\label{eq:xi-compat}\\[4pt]
\bigl(2\Phi(\xi+n+1)-\Psi(\xi+n+1)\bigr)\rho_{n+1}
&=\bigl(2\Phi(\xi+n)+\Psi(\xi+n)\bigr)\rho_n,
\quad n\in\mathbb N.\label{eq:rho-recurrence}
\end{align}
Condition \eqref{eq:xi-compat} selects the admissible base points $\xi$, and
\eqref{eq:rho-recurrence} determines $(\rho_n)_{n\geq 0}$ uniquely up to the normalisation
$\rho_0\neq0$. 

The next examples recover the Charlier polynomials already discussed in the introduction. The same algebraic architecture governs the remaining cases known in the literature, and the method set out here extends to them without conceptual difficulty. In concrete applications, the required adaptations are routine.

\begin{example}[Charlier polynomials on $\mathbb N$]\label{ex:Charlier-CaseI}
We consider the canonical pair corresponding to
\cite[Theorem~5.2, Case~I]{KLS10}, after the centred reduction described in
\cite[Observation~9.6]{CP25},
\begin{equation}\label{eq:PhiPsi-Charlier-I}
\Phi(x)=1-2\varepsilon x,\qquad
\Psi(x)=4\varepsilon x+2,\qquad
\varepsilon\in\mathbb C.
\end{equation}
We look for solutions of \eqref{eq:DS-X0} supported on a single arithmetic
progression, and therefore consider functionals of the form
\[
\mathbf u=\sum_{n=0}^\infty \rho_n\,\boldsymbol{\delta}_{\xi+n},
\quad \rho_0\neq0.
\]
The anchoring condition $c_{-1}=0$, i.e.\ $\Psi(\xi)=2\Phi(\xi)$, reduces here to
\[
4\varepsilon\xi+2=2\bigl(1-2\varepsilon\xi\bigr)=2-4\varepsilon\xi,
\]
which yields $\xi=0$ whenever $\varepsilon\neq0$. $($The degenerate case $\varepsilon=0$
is excluded by the regularity condition $d_n\neq0$ below and produces no atomic solution.$)$
For this base point, the step relation $c_n=0$ reads
\[
\bigl(2\Phi(n+1)-\Psi(n+1)\bigr)\rho_{n+1}
=
\bigl(2\Phi(n)+\Psi(n)\bigr)\rho_n.
\]
Using \eqref{eq:PhiPsi-Charlier-I},
\[
2\Phi(n)+\Psi(n)=4
\quad
2\Phi(n+1)-\Psi(n+1)=-8\varepsilon(n+1),
\]
so the recurrence simplifies to
\[
-8\varepsilon(n+1)\rho_{n+1}=4\rho_n.
\]
Solving gives
\begin{equation}\label{eq:rho-Charlier-I}
\rho_n=\rho_0\,\frac{\lambda^{\,n}}{n!},
\quad
\lambda=-\frac{1}{2\varepsilon},
\quad n\in\mathbb N,
\end{equation}
which is well defined as long as $\varepsilon\neq0$. The sequence $(\rho_n)_{n\geq 0}$ decays exponentially. Indeed, for every $m\in\mathbb N$,
\[
\sum_{n=0}^\infty |\rho_n|(1+n)^m
=|\rho_0|\sum_{n=0}^\infty \frac{|\lambda|^n}{n!}(1+n)^m<\infty,
\]
so the atomic series defines an element of $\mathcal P'$ and all termwise manipulations
above are justified. In fact, for every $R>1$,
\[
\sum_{n=0}^\infty |\rho_n|R^n
=|\rho_0|\,e^{|\lambda|R}<\infty,
\]
and since $\Phi$ and $\Psi$ have polynomial growth, the corresponding coefficients $(c_n)_{n\geq -1}$
satisfy the exponential decay condition required to apply
Corollary~\ref{cor:cnzero-strong}. Hence \eqref{eq:atomic-ansatz} with
\eqref{eq:rho-Charlier-I} indeed yields a solution of \eqref{eq:DS-X0} for the pair
\eqref{eq:PhiPsi-Charlier-I}. Finally, we discuss regularity and positivity. In the notation of \cite[Th.~9.3]{CP25}
on $X_0$, writing $\Phi(x)=ax^2+bx+c$ and $\Psi(x)=dx+e$, we have
\[
a=0,\quad b=-2\varepsilon,\quad c=1,\quad d=4\varepsilon,\quad e=2.
\]
This gives $d_n=4\varepsilon$, $e_n=-2\varepsilon n+2$, and
\[
\Phi^{[n]}(x)=\Phi(x)+\frac{n}{4}d_n=1-2\varepsilon x+n\varepsilon.
\]
The regularity criterion in \cite[Theorem~9.3]{CP25} reduces to the conditions
$d_n\neq0$ for all $n$ and
\[
\Phi^{[n]}\!\Bigl(-\frac{e_n}{d_{2n}}\Bigr)
=2+2\varepsilon n\neq0,\quad n\in \mathbb N^\times.
\]
Thus regularity holds precisely when $\varepsilon\neq0$ and
$\varepsilon\neq -1/n$ for every $n\in \mathbb N^\times$. If, in addition, $\varepsilon<0$,
then $\lambda>0$, and choosing $\rho_0=e^{-\lambda}$ in \eqref{eq:rho-Charlier-I} yields
\[
\mathbf u=\sum_{n=0}^\infty e^{-\lambda}\frac{\lambda^n}{n!}\,
\boldsymbol{\delta}_n,
\]
which coincides with the positive normalised Charlier weight on $\{0,1,2,\dots\}$,
exactly as listed in \cite[Theorem~5.2, Case~I]{KLS10}.
\end{example}

The Charlier functional admits no essentially different atomic representations: once the canonical pair $(\Phi, \Psi)$ and an admissible base point are fixed, the anchoring and step relations determine the functional uniquely up to normalisation. The various Charlier cases in \cite{KLS10} classification merely correspond to different choices of the underlying arithmetic progression.

\begin{example}[Charlier polynomials on $-\mathbb{N}$]\label{ex:Charlier-CaseIIa3}
Start from Example~\ref{ex:Charlier-CaseI} with the same atomic ansatz and anchoring/step relations, now with
\[
\Phi(x)=1-2\varepsilon x,\quad \Psi(x)=4\varepsilon x+2,\quad \varepsilon=\tfrac12.
\]
The anchoring condition selects $\xi=0$, and the step relation gives
\[
\rho_n=\rho_0\,\frac{(-1)^n}{n!},\qquad n\in\mathbb N.
\]
To place the support on the negative lattice and to obtain \emph{exactly} the Charlier weight of \cite[Theorem~5.2, Case~IIa3]{KLS10}, proceed as follows.
Apply the transpose homothety $\boldsymbol h_{-1}$ and then the geometric gauge
\[
\boldsymbol g_{-a}:\ \boldsymbol{\delta}_{-n}\longmapsto (-a)^n\,\boldsymbol{\delta}_{-n}\quad a\in\mathbb N^\times,
\]
and choose the normalisation $\rho_0=e^{-a}$. The resulting functional is
\[
\widetilde{\mathbf u}
=\sum_{n=0}^\infty e^{-a}\,\frac{a^{\,n}}{n!}\,\boldsymbol{\delta}_{-n}
=\sum_{x\in-\mathbb{N}} w_a(x)\,\boldsymbol{\delta}_x,
\quad
w_a(-n)=e^{-a}\,\frac{a^{\,n}}{n!}.
\]
Equivalently, for $x\in-\mathbb{N}$,
\[
w_a(x)=e^{-a}\,\frac{a^{-x}}{\Gamma(1-x)}.
\]
Hence \(w_a\) coincides {verbatim} with the Charlier weight on the negative lattice in \cite[Theorem~5.2, Case~IIa3]{KLS10}. Regularity is as in Example~\ref{ex:Charlier-CaseI}; see \cite[Theorem~9.3]{CP25}.
\end{example}

Since the para-Krawtchouk polynomials do not form a genuinely new family, it is worth emphasising a complementary point: if the goal were not merely to classify the sequence, but to produce an explicit representation of the underlying functional, the present scheme delivers it just as transparently as in the Charlier example, by the same atomic ansatz and the same coefficient-matching mechanism on shifted progressions.

\begin{example}[Para-Krawtchouk functionals as finite atomic solutions on two progressions]
\label{ex:paraK-atomic}
We start from the canonical reduction obtained in Example~\ref{K2}. Thus, on the
normalised lattice $X_0(s)=s$, the affinely equivalent functional
$\mathbf v\in\mathcal P'$ satisfies
\begin{equation}\label{eq:PhiPsi-paraK}
\boldsymbol{\mathrm D}_{X_0}\!\bigl(\Phi\,\mathbf v\bigr)
=
\boldsymbol{\mathrm S}_{X_0}\!\bigl(\Psi\,\mathbf v\bigr),
\quad
\Phi(x)=x^2-\gamma_0,
\quad
\Psi(x)=\alpha_0 x,
\end{equation}
where
\[
\alpha_0=-(N-1),
\quad
\gamma_0=\frac{(N-1+\gamma)(\gamma-(N-1))}{16}.
\]
We look for solutions of \eqref{eq:PhiPsi-paraK} supported on a finite union of
arithmetic progressions, and therefore consider functionals of the form
\begin{equation}\label{eq:paraK-two-string-ansatz}
\mathbf v
=
\sum_{n=0}^{M}\rho_n\,\boldsymbol{\delta}_{\xi+n}
\;+\;
\sum_{n=0}^{M'}\sigma_n\,\boldsymbol{\delta}_{\eta+n},
\quad
\rho_0\neq0,\ \sigma_0\neq0,
\end{equation}
where $\xi,\eta\in\mathbb C$ and $M,M'\in\mathbb N$ are to be determined. Proceeding exactly as in
\eqref{eq:coef-sum}--\eqref{eq:rho-recurrence}, one finds that
\eqref{eq:PhiPsi-paraK} is equivalent to the vanishing of the coefficients at the
half-shifted atoms. Assume that the two half-shifted supports are disjoint, that is,
\[
\left\{\xi+n\pm\frac12:\ n=0,1,\dots,M\right\}
\;\cap\;
\left\{\eta+n\pm\frac12:\ n=0,1,\dots,M'\right\}
=\varnothing.
\]
Under this assumption, the resulting relations decouple into two independent systems. For the $\xi$-string one obtains the anchoring condition
\begin{equation}\label{eq:paraK-anchor-xi}
\Psi(\xi)=2\Phi(\xi),
\end{equation}
together with the step relation
\begin{equation}\label{eq:paraK-step-xi}
\bigl(2\Phi(\xi+n+1)-\Psi(\xi+n+1)\bigr)\rho_{n+1}
=
\bigl(2\Phi(\xi+n)+\Psi(\xi+n)\bigr)\rho_n,
\quad n=0,1,\dots,M-1.
\end{equation}
In addition, the terminal condition at the right endpoint reads
\begin{equation}\label{eq:paraK-terminal-xi}
2\Phi(\xi+M)+\Psi(\xi+M)=0.
\end{equation}
Similarly, for the $\eta$-string one has
\begin{equation}\label{eq:paraK-anchor-eta}
\Psi(\eta)=2\Phi(\eta),
\end{equation}
and
\begin{equation}\label{eq:paraK-step-eta}
\bigl(2\Phi(\eta+n+1)-\Psi(\eta+n+1)\bigr)\sigma_{n+1}
=
\bigl(2\Phi(\eta+n)+\Psi(\eta+n)\bigr)\sigma_n,
\quad n=0,1,\dots,M'-1.
\end{equation}
together with the terminal condition
\begin{equation}\label{eq:paraK-terminal-eta}
2\Phi(\eta+M')+\Psi(\eta+M')=0.
\end{equation}

Writing $\Phi$ and $\Psi$ explicitly, \eqref{eq:paraK-anchor-xi} becomes
\[
\alpha_0\xi=2(\xi^2-\gamma_0),
\]
so that $\xi$ must be a root of $2x^2-\alpha_0x-2\gamma_0$. In the present
parameters one has the identity
\[
\alpha_0^2+16\gamma_0=\gamma^2,
\]
and therefore the two admissible base points are
\begin{equation}\label{eq:paraK-xi-roots}
\xi_\pm=\frac{\alpha_0\pm\gamma}{4}
=
\frac{-(N-1)\pm\gamma}{4}.
\end{equation}
The same computation applies to \eqref{eq:paraK-anchor-eta}, hence
$\eta\in\{\xi_+,\xi_-\}$. We fix the two progressions by taking
\begin{equation}\label{eq:paraK-choice-xi-eta}
\xi=\xi_+,
\quad
\eta=\xi_-.
\end{equation}
For a fixed admissible base point $\xi$, the step relation \eqref{eq:paraK-step-xi}
determines $(\rho_n)_{n\geq 0}$ uniquely up to the normalisation $\rho_0$, provided that
\[
2\Phi(\xi+n+1)-\Psi(\xi+n+1)\neq 0,
\quad n=0,1,\dots,M-1.
\]
Under this non-degeneracy condition, \eqref{eq:paraK-step-xi} can be rewritten as
\[
\rho_{n+1}
=
\frac{2\Phi(\xi+n)+\Psi(\xi+n)}
     {2\Phi(\xi+n+1)-\Psi(\xi+n+1)}\,\rho_n,
\quad n=0,1,\dots,M-1,
\]
and hence
\begin{equation}\label{eq:paraK-rho-product}
\rho_n
=
\rho_0\,
\prod_{k=0}^{n-1}
\frac{2\Phi(\xi+k)+\Psi(\xi+k)}
     {2\Phi(\xi+k+1)-\Psi(\xi+k+1)},
\quad n=1,\dots,M.
\end{equation}
Similarly, \eqref{eq:paraK-step-eta} determines $(\sigma_n)_{n\geq 0}$ uniquely up to
$\sigma_0$, provided that
\[
2\Phi(\eta+n+1)-\Psi(\eta+n+1)\neq 0,
\quad n=0,1,\dots,M'-1,
\]
and then yields
\begin{equation}\label{eq:paraK-sigma-product}
\sigma_n
=
\sigma_0\,
\prod_{k=0}^{n-1}
\frac{2\Phi(\eta+k)+\Psi(\eta+k)}
     {2\Phi(\eta+k+1)-\Psi(\eta+k+1)},
\quad n=1,\dots,M'.
\end{equation}
The para-Krawtchouk situation is finite already at the level of the second-order
equation (cf.\ $b(0)=0$ and $a(N-1)=0$ in
Example~\ref{ex:paraKrawtchouk-1classical-form}). Under the affine reduction leading
to \eqref{eq:PhiPsi-paraK}, finiteness translates into the existence of indices
$M,M'$ such that the terminal conditions \eqref{eq:paraK-terminal-xi} and
\eqref{eq:paraK-terminal-eta} hold, which forces the truncation of the two strings.
Thus, for admissible parameters, the ansatz \eqref{eq:paraK-two-string-ansatz} with base
points \eqref{eq:paraK-choice-xi-eta} and weights
\eqref{eq:paraK-rho-product}--\eqref{eq:paraK-sigma-product} yields a finite atomic
solution of \eqref{eq:PhiPsi-paraK}, unique up to the normalisations $\rho_0$ and
$\sigma_0$.
Finally, recall from Example~\ref{K2} that $\mathbf v$ and the original
para-Krawtchouk functional $\mathbf u$ are related by
\[
\mathbf v=
\left(\boldsymbol{\tau}_{-\frac{N-1+\gamma}{4}}\circ \boldsymbol{h}_{\frac12}\right)\mathbf u,
\]
and, equivalently,
\[
\mathbf u=
\left(\boldsymbol{h}_{2}\circ \boldsymbol{\tau}_{\frac{N-1+\gamma}{4}}\right)\mathbf v.
\]
Therefore, the above finite atomic representation of $\mathbf v$ transports
immediately to a finite atomic representation of $\mathbf u$ on the lattice
$X(s)=2s$, and the associated orthogonal polynomials are precisely the
para-Krawtchouk family of Example~\ref{ex:paraKrawtchouk-1classical-form}.
\end{example}

\section*{Acknowledgements}
The authors acknowledge financial support from the Centre for Mathematics of the University of Coimbra (CMUC), funded by the Portuguese Foundation for Science and Technology (FCT), under the projects UID/00324/2025 (\url{https://doi.org/10.54499/UID/00324/2025}) and UID/PRR/00324/2025.
 The first author acknowledges financial support from the FCT under the grant \url{https://doi.org/10.54499/2022.00143.CEECIND/CP1714/CT0002}.
 The second author acknowledges financial support from FCT under the grant DOI: 10.54499/UI.BD. 154694.2023.
 
 \appendix

\bibliographystyle{amsplain}  

\bibliography{bib}  
\end{document}